\documentclass{amsart}
 
\usepackage{amsmath}
\usepackage{amsthm,amssymb,color,comment}
\usepackage{bbm}
\usepackage{hyperref}
\usepackage[TS1,T1]{fontenc}
\usepackage[utf8]{inputenc}
\usepackage{dsfont}
\usepackage{tikz}
\usepackage{enumitem}
\usepackage{subfigure}

\numberwithin{equation}{section}

\newtheorem{thm}{Theorem}[section]

\newtheorem{lem}[thm]{Lemma}

\theoremstyle{definition}
\newtheorem{Ass}[thm]{Assumption}
\newtheorem{rem}[thm]{Remark}
\newtheorem{cor}[thm]{Corollary}
\newtheorem{Nott}[thm]{Notation}

\DeclareMathOperator{\DIV}{div}
\DeclareMathOperator*{\wstlim}{w^*-lim}

\newcommand{\R}{\mathbb{R}}

\newcommand{\supp}{\text{supp}}
\newcommand{\diff}{\mathop{}\!\mathrm{d}}
\newcommand{\wstar}{\overset{\ast}{\rightharpoonup}}

\newcommand{\ueps}{u^{\varepsilon}}
\newcommand{\veps}{v^{\varepsilon}}
\newcommand{\weps}{w^{\varepsilon}}
\newcommand{\zeps}{z^{\varepsilon}}

\newcommand{\weaks}{\overset{\ast}{\rightharpoonup}}
\newcommand{\weaksinv}{\overset{\ast}{\leftharpoonup}}
\newcommand{\weak}{\rightharpoonup}

\makeatletter
\newcommand{\doublewidetilde}[1]{{%
  \mathpalette\double@widetilde{#1}%
}}
\newcommand{\double@widetilde}[2]{%
  \sbox\z@{$\m@th#1\widetilde{#2}$}%
  \ht\z@=.9\ht\z@
  \widetilde{\box\z@}%
}
\makeatother
\author{Beno\^\i t Perthame}
\address{Sorbonne Universit{\'e}, CNRS, Universit\'{e} de Paris, Inria, Laboratoire Jacques-Louis Lions UMR7598, F-75005 Paris}
\email{Benoit.Perthame@sorbonne-universite.fr}
\thanks{Beno\^\i t Perthame has received funding from the European Research Council (ERC) under the European Union's Horizon 2020 research and innovation program (grant agreement No 740623).}

\author{Jakub Skrzeczkowski}
\address{Faculty of Mathematics, Informatics and Mechanics, University of Warsaw, Stefana Banacha 2, 02-097 Warsaw, Poland}
\email{jakub.skrzeczkowski@student.uw.edu.pl}
\thanks{Jakub Skrzeczkowski was supported by National Science Center, Poland through project no. 2018/30/M/ST1/00423. This work was completed while J.S. was a visitor at Laboratoire Jacques-Louis Lions whose kind hospitality he appreciates.}

\begin{document}

\title[Fast reaction limit with nonmonotone reaction function]{Fast reaction limit with nonmonotone reaction function}

\begin{abstract}
We analyse fast reaction limit in the reaction-diffusion system with nonmonotone reaction function and one non-diffusing component. As speed of reaction tends to infinity, the concentration of non-diffusing component exhibits fast oscillations. We identify precisely its Young measure which, as a by-product, proves strong convergence of the diffusing component, a result that is not obvious from a priori estimates. Our work is based on analysis of regularization for forward-backward parabolic equations by Plotnikov. We rewrite his ideas in terms of kinetic functions which clarifies the method, brings new insights, relaxes assumptions on model functions and provides a weak formulation for the evolution of the Young measure.
\end{abstract}



\maketitle
\setcounter{tocdepth}{1}
\tableofcontents
\section{Introduction}
\noindent Let $\Omega \subset \R^d$ be a smooth, bounded domain. We consider the following system of reaction-diffusion equations with Neumann boundary conditions, 
\begin{align}
\partial_t \ueps &=  \frac{\veps - F(\ueps)}{\varepsilon}, \label{system_1}\\
\partial_t \veps &= \Delta \veps + \frac{F(\ueps) - \veps}{\varepsilon}, \label{system_2}
\end{align}
where $t \geq 0$, $x \in \Omega$ and $F:\R \to [0,\infty)$ is a sufficiently smooth function. System \eqref{system_1}--\eqref{system_2} with a non-monotonic $F$, which is our interest here, is an interesting toy model for studying oscillations in reaction-diffusion systems as they are known to occur in its steady states \cite{MR3908864}.

\begin{Ass}[Initial data] \label{ass:ID}
 The system is completed with initial values $\ueps(0,x) = u_0(x)$, $\veps(0,x) = v_0(x)$ satisfying 
\begin{enumerate}
\item Nonnegativity: $u_0, v_0 \geq 0$.
\item \label{ass:ic2} Regularity: $u_0, v_0 \in C^{2+\alpha}(\overline{\Omega})$ for some $\alpha \in (0,1)$.
\item Boundary condition: $u_0, v_0$ satisfy the Neumann boundary condition.
\end{enumerate}
\end{Ass}
\noindent Under appropriate assumptions (see Theorem \ref{theorem:well_posed}), there is a unique classical solution of \eqref{system_1}--\eqref{system_2} which is bounded and nonnegative. Such systems are usually called {\it mass conservative} as it is easy to check that the quantity
$$
\int_{\Omega} u(t,x) + v(t,x) \diff x
$$
remains constant. Such equations have been used to model biological and chemical phenomena including cell polarity regularization (assymetric organization of cellular structures) \cite{MR2367252} and they received a lot of mathematical attention \cite{MR2646071,MR3071451} in particular for their pattern formation ability related to Turing instability \cite{morita2016reaction,MR3908864}. Moreover, systems with one non-diffusive component are widely studied in the literature, serving as models for early carcinogenesis \cite{MR3039206} and also for pattern formation \cite{MR3329327, MR3600397}. \\

\noindent  Our interest lies in the so-called fast reaction limit corresponding to $\varepsilon \to 0$.   By now, this problem is fairly classical assuming that reaction function $F$ is monotone \cite{MR2009623}. In this spirit, fast reaction limits have been studied for a great variety of reaction-diffusion systems, also with more than two components \cite{MR3005532, MR4040718,MR2776460} or reaction-diffusion equation coupled with an ODE \cite{MR3655798}. They usually lead to the cross-diffusion systems where the gradient of one quantity induces a flux of another one \cite{MR2251792}, a phenomena that is non-negligible for instance in chemistry \cite{vanag2009cross} and is constantly studied from the mathematical point of view, see \cite{MR4051984,MR3783102} and references therein. A slightly different type of problem deals with the fast-reaction limit for irreversible reactions which leads to free boundary problems  \cite{MR3501846,MR589954}.
%MR1687440,MR1383226
We refer the reader to \cite{MR3905643,murakawa2019fast} and references therein for further details and another limits in reaction-diffusion systems. \\
%When it comes to results with $F$ being not necessarily monotone, it was shown in \cite{MR3908864} that under some special cases concerning values of diffusion coefficients one can obtain similar cross-diffusion systems using refined energy estimates from \cite{morita2016reaction}. \\

\noindent To focus our attention, we consider functions $F$ with particular monotonicity profile as plotted in Fig.~\ref{plot:Fd}. Our first result asserts that, up to a subsequence,
\begin{equation}\label{eq:our_main_result}
F(\ueps), \; \veps \to v \mbox{ in } L^2((0,T)\times\Omega), \qquad \qquad \ueps \weaks u:= \sum_{i=1}^3 \lambda_i \, S_i(v) \mbox{ in } L^{\infty}((0,T)\times\Omega),
\end{equation}
where the weights $\lambda_1(t,x)$, $\lambda_2(t,x)$ and $\lambda_3(t,x)$ are nonnegative numbers such that $\sum_{i=1}^3 \lambda_i = 1$ while $S_1$, $S_2$ and $S_3$ are three possible inverses of $F$ defined in Notation \ref{intro:not_inv_phi}, cf. Fig.~\ref{plot:Fd}. Our main result is however to derive a kinetic  equation for the weights. On the one hand, strong convergence of $\veps$ is surprising as its compactness in time does not seem to be available from a priori estimates. On the other hand, weak$^*$ limit of $\ueps$ can be interpreted as the weak form of the identity $v = F(u)$ known from the classical fast reaction limits. However, in our case, mass of $u$ splits for three parts assosciated to the preimages of $v$ under the map $F$. This intuition is made more precise in Theorem
~\ref{thm_main_id_YM} using the language of Young measures.\\
%%%%%%%%%%%%%%%%%%%%%%%%%%%%%%%%%%%
%%%%%%%%%%%%%%%%%%%%%%%%%%%%%%%%%%%
\begin{figure}
\begin{tikzpicture}

%coordinate system and dashed lines
\draw[line width=0.4mm,->] (0,0) -- (8.5,0) node[anchor=north west] {$u$};
\draw[line width=0.4mm,->] (0,0) -- (0,5) node[anchor=south east] {$F(u)$};

%plot of the function
\draw [black] plot [smooth, tension=1] coordinates { (0,0.0) (0.45,1.5) (2,3.0) (4,1.5) (6,3.0) (7.5, 4.5)};

%dashed lines with f- and f+
\draw [dashed] (0,3) -- (8.5,3);
\draw [dashed] (0,1.5) -- (8.5,1.5);
\node at (-0.3, 3) {$f_{+}$};
\node at (-0.3, 1.5) {$f_{-}$};

%nodes on x axis
\node at (0.45,-0.5) {$\alpha_{-}$};
\node at (2,-0.5) {$\alpha_{+}$};
\node at (4,-0.5) {$\beta_{-}$};
\node at (6,-0.5) {$\beta_{+}$};
\draw (0.45,-0.2) -- (0.45,0.2);
\draw (2,-0.2) -- (2,0.2);
\draw (4,-0.2) -- (4,0.2);
\draw (6,-0.2) -- (6,0.2);

\end{tikzpicture}
\vspace{-4mm}
\caption{Plot of a typical function $F$. It is strictly increasing in intervals $(-\infty, \alpha_{+}) \cup (\beta_{-}, \infty)$ and strictly decreasing in $(\alpha_{+}, \beta_{-})$. For $r \in (f_{-}, f_{+})$, the function $F$ is not invertible and equation $F(u) = r$ has three roots $s = S_1(r)\leq S_2(r) \leq S_3(r)$.}
\label{plot:Fd}
\end{figure}
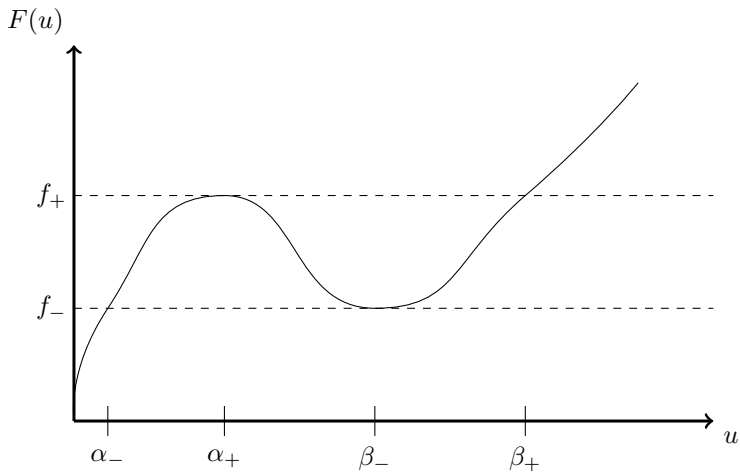

\noindent Our strategy to prove \eqref{eq:our_main_result} is to combine ideas from kinetic formulations of PDEs \cite{MR2064166,MR1099693} and from the insightful work of Plotnikov \cite{MR1299852} (see also \cite{MR2103092, MR1015926,MR1620640, MR2429862} for similar problems). He considered the following regularization:
\begin{equation}\label{intro:regularizing_seq}
\partial_t w^{\varepsilon} = \Delta A(w^{\varepsilon}) + \varepsilon\, \Delta (\partial_tw^{\varepsilon})
\end{equation}
of the ill-posed problem $\partial_t w = \Delta A(w)$ where $A$ is assumed to have a similar monotonicity profile as in Fig.~\ref{plot:Fd}. Plotnikov studied the limit of $w^{\varepsilon}$ as $\varepsilon \to 0$. Using the theory of Young measures, he was able to predict oscillations in the limit and obtain the similar characterization of the limit $w$ as $u$ in our result \eqref{eq:our_main_result}. We comment more on connection between our work and Plotnikov paper in Section \ref{sect:connection_with_Plotnikov}. We remark that analysis of $\partial_t \weps = \Delta A(\weps)$ with non-necessarily monotone function $A$ is constantly receiving attention in mathematical community \cite{MR3466547,MR2103092,MR2563628,MR2765690}.\\
 
\noindent Unlike the original work of Plotnikov, in the process of limit identification, we exploit kinetic formulation. This is a well-known concept for scalar conservation laws \cite{MR1201239,MR2064166} that brought some connections with kinetic equations \cite{MR1099693} and degenerate parabolic equations \cite{MR1981403,GST2019}. Although this is an approach equivalent with Young measures, working directly on functions is simpler as limit identification is based on a certain functional identity cf. Theorem
~\ref{thm:analytic_char_result}. This approach results in a PDE satisfied by the kinetic functions cf.~\eqref{thm:2ndPDE_kinfunct} which provides some information on evolution of weights $\lambda_i$ in \eqref{eq:our_main_result}, cf. Section \ref{sect:diff_ineq_weights}. We comment more on connections between our work and Plotnikov's paper in Section \ref{sect:connection_with_Plotnikov}. \\

\noindent In this paper, we discuss the limit of \eqref{system_1}--\eqref{system_2} when $\varepsilon \to 0$. First, we present a priori estimates (Section \ref{sect:aprioriest}). Then, in Section \ref{sect:main_kin_form}, we introduce kinetic formulation which allows to prove \eqref{eq:our_main_result} in Section \ref{sect:prf_main_result}. In Section \ref{sect:diff_ineq_weights}, we use kinetic formulation to derive some formal differential equations for coefficients $\lambda_i$ in \eqref{sect:diff_ineq_weights}. The two last subsections are devoted to discuss how our work is related to the Plotnikov's paper and present some open problems in the field. \\

\noindent We list the main novelties of our work below.
\begin{itemize}
    \item We rewrite Plotnikov's method in terms of kinetic functions and identify limits of $\ueps$ and $\veps$ in system \eqref{system_1}--\eqref{system_2}.
    \item   We establish the PDE \eqref{eq:PDE2_sat_by_kin_LIMIT} satisfied by limiting kinetic functions. The latter can be viewed as a weak formulation for equation \eqref{eq:ODE_combined_noloc} satisfied, when $v$ is smooth, by weights $\lambda_1$, $\lambda_2$ and $\lambda_3$, see Section \ref{sect:diff_ineq_weights}.
    \item We modify Plotnikov's method by exploiting natural energy for \eqref{system_1}--\eqref{system_2} rather than Plotnikov's variables. Therefore, we do not need to assume $F'(u)>-1$, see Section \ref{sect:connection_with_Plotnikov}.
\item We relax the nondegeneracy condition on the nonlinearity $F$, see \eqref{ass:nondeg} in Assumption~\ref{ass:nonlinearity}.
\end{itemize}

%-----------------------------------------
\section{Assumptions and the main results}
%-----------------------------------------

\noindent Before we start, let us precisely formulate our assumptions and notations for inverses of function $F$.
\begin{Nott}\label{intro:not_inv_phi}
Let $S_1(\lambda) \leq S_2(\lambda) \leq S_3(\lambda)$ be the solutions of equation $F(S_i(\lambda)) = \lambda$ as already introduced in \eqref{ass:nondeg} in Assumption \ref{ass:nonlinearity} (see Fig.~\ref{plot:Fd}). These are inverses of $F$ satisfying
$$
S_1:(-\infty,f_{+}) \to (-\infty,\alpha_{+}), \qquad S_2:(f_{-},f_{+}) \to (\alpha_{+},\beta_{-}), \qquad S_3:(f_{-},\infty) \to (\beta_{-},\infty).
$$
Their role is too focus analysis on parts of the plot of $F$ where the monotonicity of $F$ does not change. By a small abuse of notation, we extend functions $S_i$ by a constant value to the whole of~$\R$. We usually write, for images of functions $S_1$, $S_2$, $S_3$ and for their domains
$$
I_1 = (-\infty,\alpha_{+}), \qquad \qquad I_2 = (\alpha_{+},\beta_{-}), \qquad \qquad I_3 = (\beta_{-}, \infty),
$$
$$
J_1 = (-\infty,f_{+}), \qquad \qquad J_2 = (f_{-},f_{+}), \qquad \qquad J_3 = (f_{-}, \infty) .
$$
\end{Nott}
\begin{Ass}[Reaction function $F$] \label{ass:nonlinearity}
We assume that the function $F(u)$ satisfies:
\begin{enumerate}
\item Regularity,  nonnegativity: $F \in C^1\big(\R; [0,\infty) \big)$, with $F(0) = 0$.
\item \label{ass:mon} Piecewise monotonicity of $F$: there are $\alpha_{-}<\alpha_{+}<\beta_{-} < \beta_{+}$ such that $F(\beta_{-}) = F(\alpha_{-})$, $F(\alpha_+) = F(\beta_+)$, $F$ is strictly increasing on $(-\infty,\alpha_+) \cup (\beta_{-}, \infty)$ and strictly decreasing on $(\alpha_{+}, \beta_{-})$ (see Fig.~ \ref{plot:Fd}). Moreover, $\lim_{u\to \infty} F(u) = \infty$.
\item \label{ass:nondeg} Nondegeneracy: 
%for $r \in (f_{-}, f_{+})$, 
%, let $S_1(r), S_2(r), S_3(r)$ be the three roots of equation $F(S_i(r)) = r$ such that $S_1(r) < S_2(r) < S_3(r)$; then the maps $1 + S_1'(r), 1 + S_2'(r),1 + S_3'(r)$ are linearly independent in each subinterval of $(f_-,f_+)$.
in all subintervals of $(f_-,f_+)$, the vanishing linear combination $$\sum_{i=1}^3 a_i\;  \big( S_1'(r)+1\big) =0 $$ implies $a_1 +a_2+a_3=0$. 
\end{enumerate}
\end{Ass}
\noindent Let us comment on the nondegeneracy condition \eqref{ass:nondeg} that is by no means an innocent assumption. For instance, it holds true if the functions $1 + S_1'(r), 1 + S_2'(r),1 + S_3'(r)$ are linearly independent in each subinterval of $(f_-,f_+)$ which was the original assumption made by Plotnikov \cite{MR1299852}. On the other hand, this condition excludes piecewise linear functions $F$. Nevertheless, it is usually made in this type of problems \cite{MR657784, MR1015926,MR1299852}. For sufficiently smooth functions, a typical approach to check nondegenracy condition is computing Wronskian of these functions \cite[Section 1.3]{MR1721985}, see \cite[Proposition 2]{MR1620640} for a particular example. We list it as a one of the open problems in Section
~\ref{sect:open_problems} to relax the nondegeneracy condition.\\

\noindent A first result of this paper reads:
\begin{thm}[Limits for $\veps, \; \ueps$]\label{thm_main_id_YM}
Let $T>0$ and $(\ueps, \veps)$ be the solution of \eqref{system_1}--\eqref{system_2}. Then, up to a subsequence, $\ueps \wstar u$ weakly$^*$ in $L^{\infty}((0,T)\times\Omega)$ and $F(\ueps), \; \veps \to v$ strongly in $L^2((0,T)\times\Omega)$. Moreover, the Young measure generated by $\{\ueps\}_{\varepsilon \in (0,1)}$ is a convex combination of Dirac masses
\begin{equation}\label{eq:repr_main_result}
\mu_{t,x} = \lambda_1(t,x)\,\delta_{S_1(v(t,x))} + \lambda_2(t,x)\,\delta_{S_2(v(t,x))} + \lambda_3(t,x)\,\delta_{S_3(v(t,x))},
\end{equation}
where $S_1$, $S_2$ and $S_3$ are the inverses of $F$ defined in Notation \ref{intro:not_inv_phi} while $\lambda_1$, $\lambda_2$, $\lambda_3$ are nonnegative numbers such that $\sum_{i=1}^3 \lambda_i~=~1$.
\end{thm}
\noindent The proof is presented in Section~\ref{sect:prf_main_result}. Loosely speaking, representation~\eqref{eq:repr_main_result} means that for small values of $\varepsilon$, the function $u(t,x)$ should oscillate between (at most) three values. This is observed by  numerical simulations in Fig.~\ref{fig:discont}. In fact the unstable state is reached only during the transient.\\
%\begin{figure}
%\subfigure{\includegraphics[width=7.5cm]{pic1_osc_s.png}}
%\subfigure{\includegraphics[width=7.5cm]{pic2_osc_s.png}}
%\vspace{-5mm}
%\caption{Evolution of $\ueps(t,x)$ solving \eqref{system_1}--\eqref{system_2} in one space %dimension. Here, a small value of $\varepsilon$ is chosen and $F$ is a cubic function of the %form $F(u) = u^3 - Au^2 + Bu$ for appropriately chosen $A$ and $B$ so that $F$ has monotonicity %profile as in Fig.~\ref{plot:Fd}. Two time shots are presented to show how oscillations in %$u(t,x)$ are formed.}
%\label{fig:oscillations}
%\end{figure}

\noindent The connection between $u$ and $v$ in Theorem \ref{thm_main_id_YM} is formulated in the language of Young measures and reader not familiar with this topic is referred to \cite{MR1034481} for a concise introduction with applications. Briefly speaking, Young measures allow to represent weak limits of nonlinear functions. More precisely, let $\{\mu_{t,x}\}_{t,x}$ and $\{\nu_{t,x}\}_{t,x}$ be the Young measures generated by sequences $\{\ueps\}_{\varepsilon \in (0,1)}$ and $\{\veps\}_{\varepsilon \in (0,1)}$ respectively. Then, for any bounded function $G: \R \to \R$ we have (up to a subsequence and for a.e. $(t,x) \in (0,T)\times \Omega$)
$$
G(\ueps) \wstar \int_{\R} G(\lambda) \diff \mu_{t,x}(\lambda) := \langle G, \mu_{t,x} \rangle, \qquad G(\veps) \wstar \int_{\R} G(\lambda) \diff \nu_{t,x}(\lambda) := \langle G, \nu_{t,x} \rangle.
$$
The proof of Theorem \ref{thm_main_id_YM} goes as follows. One rewrites equations \eqref{system_1}--\eqref{system_2} in terms of kinetic functions. Using compensated compactness \cite{MR894077,MR584398}, we obtain Lemma~\ref{lem:compactness_product} and the functional identity for kinetic functions~\eqref{thmeq:main_identity} from which we deduce the kinetic function shape for $v(t,x)$ in Section~\ref{sect:prf_main_result}. This implies that the Young measure generated by $\{\veps\}_{\varepsilon \in (0,1)}$ is a Dirac mass which proves
~\eqref{eq:repr_main_result}.
\\
\noindent The crucial step in the proof of Theorem~\ref{thm_main_id_YM}, which is a new result by its own, is a PDE satisfied by the kinetic functions generated by $\{\ueps\}_{\varepsilon \in (0,1)}$ and $\{\veps\}_{\varepsilon \in (0,1)}$. 

\begin{thm}[Kinetic PDE]\label{thm:intro_sec_PDE_kin_fun}
Let $p, q: (0,T)\times \Omega \times \R \to \R$ be the $L^\infty$-weak$^*$ limits (up to extraction of subsequences) as below:
$$
p(t,x,\xi) = \wstlim_{\varepsilon\to 0}\, \mathds{1}_{0\leq \xi \leq \ueps(t,x)}(\xi), \qquad \qquad q(t,x,\xi) = \wstlim_{\varepsilon\to 0}\, \mathds{1}_{0\leq \xi \leq \veps(t,x)}(\xi).
$$
Then, there is a bounded nonnegative measure $n$ on $(0,T)\times\Omega\times\R$ such that equation
\begin{equation}\label{eq:PDE2_sat_by_kin_LIMIT_intro}
\partial_t \left[\int_{\R} p(t,x,\lambda) \, \delta_{\xi = F(\lambda)}(\xi) \diff \lambda +  q(t,x,\xi)\right] - \Delta_x q(t,x,\xi) =  \partial_{\xi}\, {n}(t,x,\xi)
\end{equation}
holds in the sense of distributions.
\end{thm}
\noindent Theorem \ref{thm:intro_sec_PDE_kin_fun} is proved in Section \ref{sect:intro_to_knetic_func} (part of Theorem \ref{thm:2ndPDE_kinfunct}). As a consequence, we can formulate equations for evolution of weights $\{\lambda_i(t,x)\}_{i=1,2,3}$ proved in Secton \ref{sect:diff_ineq_weights}. 
\begin{thm}[Equations for weights]\label{thm:intro_eqns_weights}
Let $\{\lambda_i(t,x)\}_{i=1,2,3}$ be as in \eqref{eq:repr_main_result}. We set
$$
\kappa_1(t,x) = 1 - \lambda_1(t,x), \qquad \qquad \kappa_2(t,x) = \lambda_3(t,x).
$$ 
\begin{enumerate}
    \item\label{thm:part1_eq_weights} Suppose additionally that sequences $\{\partial_t \veps\}_{\varepsilon \in (0,1)}$ and $\{\Delta \veps\}_{\varepsilon \in (0,1)}$ are uniformly bounded in $L^2((0,T)\times\Omega)$. Then, we have $\partial_t \lambda_i(t,x) = 0$ for $i=1,2,3$ and $(t,x) \in \mathcal{O}$ where $\mathcal{O}\subset (0,T)\times\Omega$ is any open set where $v(t,x)$ is continuous. In particular, no splitting of mass may occur.
    \item\label{thm:part2_eq_weights} In general, if $$\mathcal{O} \subset \{(t,x,\xi): f_{-}< v(t,x) < \xi_0 < \xi < f_{+} \}
    $$
    is an open set for some $\xi_0$, we have for $(t,x,\xi) \in \mathcal{O}$
    \begin{equation}\label{eq:intro_thm_PDE_weigts_1}
    \partial_t \int_{S_1(v(t,x))}^{S_2(v(t,x))} \kappa_1(t,x) \, \delta_{\xi = F(\lambda)}(\xi) \diff \lambda  =  \partial_{\xi} {n}(t,x,\xi),
    \end{equation}
    where $n$ is a nonnegative measure from Theorem \ref{thm:intro_sec_PDE_kin_fun}. Similarly, if 
    $$\mathcal{O} \subset \{(t,x,\xi): f_{-} < \xi < \xi_0 < v(t,x) < f_{+} \} 
    $$
    is an open set for some $\xi_0$, we have for $(t,x,\xi) \in \mathcal{O}$
    \begin{equation}\label{eq:intro_thm_PDE_weigts_2}
    \partial_t \int_{0}^{S_1(v(t,x))}  \delta_{\xi = F(\lambda)}(\xi) \diff \lambda  + \partial_t \int_{S_2(v(t,x))}^{S_3(v(t,x))} \kappa_2(t,x) \, \delta_{\xi = F(\lambda)}(\xi) \diff \lambda = \partial_{\xi} {n}(t,x,\xi).
    \end{equation}
\end{enumerate}
\end{thm}
\noindent Part \eqref{thm:part1_eq_weights} of Theorem \ref{thm:intro_eqns_weights} implies that if $\ueps$ oscillates between two states in some subset, function $\veps$ should form a discontinuity there. This phenomenon is presented in Fig.~\ref{fig:discont}.\\
\begin{figure}
\subfigure{\includegraphics[width=7.5cm]{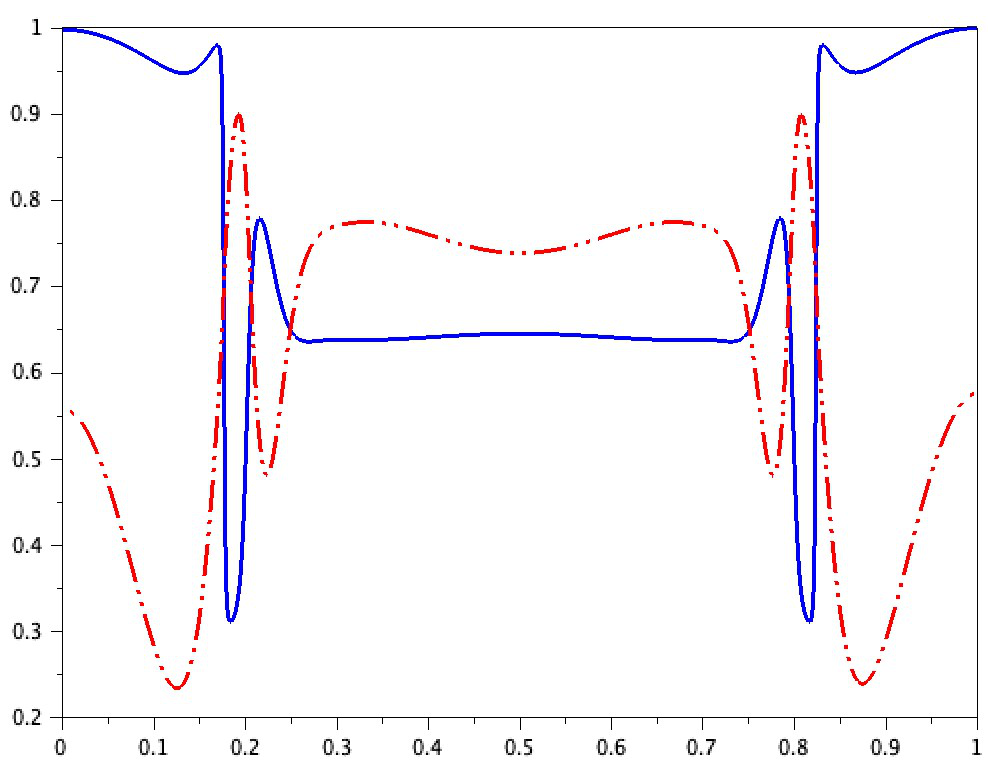}}
\subfigure{\includegraphics[width=7.5cm]{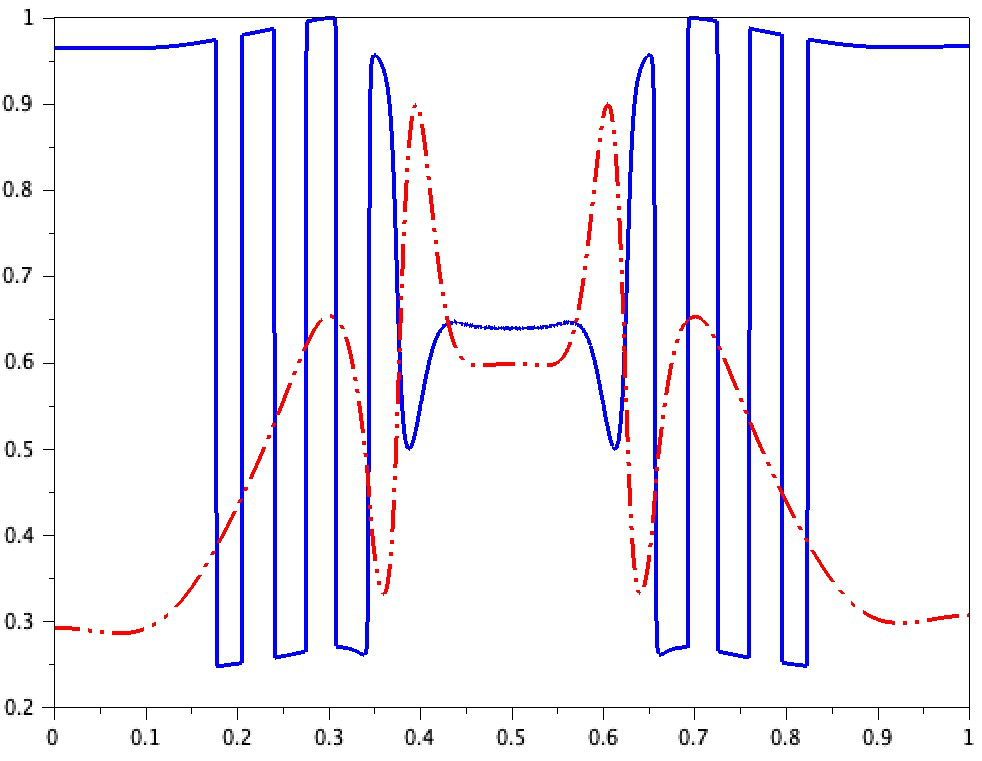}}
\vspace{0mm}
\caption{Evolution of $\ueps$ (continuous line) and $\veps$ (dash-dotted line) solving \eqref{system_1}--\eqref{system_2} in one space dimension with fixed and small value of $\varepsilon>0$. Two time shots are presented to show dependence between oscillations of $\ueps$ and $\veps$. When $\ueps$ oscillates, $\veps$ also exhibits oscillatory behaviour. However, when the weights in the equation \eqref{eq:repr_main_result} stabilize and only one of them is not vanishing, oscillations of $\veps$ disappear.}
\label{fig:discont}
\end{figure}

\noindent The main tool to prove Theorem \ref{thm:intro_sec_PDE_kin_fun} is the following energy equality. Given a smooth test function $\phi: \R \to \R$, we define
\begin{equation}\label{eq:def_of_Psi_Phi}
\Psi(\lambda) := \int_0^{\lambda} \phi(F(\tau)) \diff \tau, \qquad \qquad 
\Phi(\lambda) := \int_0^{\lambda} \phi(\tau) \diff \tau.
\end{equation}
Multiplying equation \eqref{system_1} with $\phi(F(\ueps))$ and equation \eqref{system_2} with $\phi(\veps)$ we obtain
\begin{align*}
\partial_t \Psi(\ueps) &=  \frac{\veps - F(\ueps)}{\varepsilon} \, \phi(F(\ueps)),\\
\partial_t \Phi(\veps) &= \Delta \Phi(\veps) - \phi'(\veps)\,|\nabla \veps|^2 + \frac{F(\ueps) - \veps}{\varepsilon}\, \phi(\veps).
\end{align*}
Summing up these equations we deduce
\begin{equation}\label{eq:PDE_for_kinetic_function_withTF}
\partial_t \Psi(\ueps) + \partial_t \Phi(\veps) = \Delta \Phi(\veps) - \phi'(\veps)\,|\nabla \veps|^2 - \frac{\big(\veps - F(\ueps)\big)\, (\phi(\veps) - \phi(F(\ueps))\big)}{\varepsilon}.
\end{equation}
This energy equality provides a priori estimates stated in Theorem~\ref{theorem:well_posed}, the PDE for the kinetic functions in Theorem~\ref{thm:2ndPDE_kinfunct} and compensated compactness results (Lemma \ref{lem:compactness_product}) necessary to derive the functional identity~\eqref{thmeq:main_identity}.

%---------------------------
\section{A priori estimates}
\label{sect:aprioriest}
%---------------------------

\noindent As long as $\varepsilon$ is positive, solutions of system \eqref{system_1}--\eqref{system_2} are smooth, however some estimates, which are instrumental for studying the oscillatory limit, are uniform in  $\varepsilon$.

\begin{thm}\label{theorem:well_posed}
%Let $T > 0$. 
There exists the unique classical solution $\ueps, \veps: [0,\infty)\times \Omega \to \R$  of \eqref{system_1}--\eqref{system_2} which is nonnegative and has regularity
$$
\ueps \in C^{\alpha, 1 + \alpha/2}\left([0,\infty) \times \overline{\Omega}\right), \qquad \veps \in C^{2+\alpha, 1 + \alpha/2}\left([0,\infty) \times \overline{\Omega}\right).
$$
Moreover, we have
\begin{enumerate}
\item \label{PUB0}  $0 \leq \ueps \leq M $, $0 \leq \veps \leq M $ with $M = \max(\|F(u_0)\|_{\infty},\, \|u_0\|_{\infty}, \, \|v_0\|_{\infty},\,f_{+},\,\beta_{+})$,
%$\{ \ueps\}_{\varepsilon \in (0,1)}$, $\{ \veps\}_{\varepsilon \in (0,1)}$ are uniformly bounded in $L^{\infty}((0,T)\times \Omega)$,
\item \label{PUB1} $\{\nabla \veps \}_{\varepsilon \in (0,1)}$ is uniformly bounded in $L^2((0,\infty)\times \Omega)$,
\item \label{PUB3} $\left\{\frac{F(\ueps) - \veps}{\sqrt{\varepsilon}}\right\}_{\varepsilon \in (0,1)}$ and $\left\{\sqrt{\varepsilon}\, \Delta \veps \right\}_{\varepsilon \in (0,1)}$ are uniformly bounded in $L^2((0,\infty)\times \Omega)$.
\end{enumerate}
\end{thm}
\begin{proof}
Because the right hand sides are locally Lipschitz continuous, local existence, nonnegativity and uniqueness follows from standard theory.
%cf. \cite[Theorem~1, p. 111]{MR755878}. 
To prove global existence, we need to establish uniform bounds as in~\eqref{PUB0}. We consider smooth and nondecreasing test function $\phi: \R \to \R$ as well as $\Psi$ and $\Phi$ defined by \eqref{eq:def_of_Psi_Phi}. Then, it follows from \eqref{eq:PDE_for_kinetic_function_withTF} that
\begin{equation*}
\partial_t \Psi(\ueps) + \partial_t \Phi(\veps) \leq \Delta \Phi(\veps)
\end{equation*}
thanks to the monotonicity of $\phi$. Therefore, the nonnegative map
\begin{equation}\label{eq:functions_Linf_bounds}
t \mapsto \int_{\Omega} \big[\Psi(\ueps(t,x)) +\Phi(\veps(t,x))\big] \diff x
\end{equation}
is nonincreasing. 
We choose $\phi$ such that $\phi = 0$ on $[0,M]$ and $\phi' > 0$ on $(M,\infty)$. Then, the map in \eqref{eq:functions_Linf_bounds} vanishes at $t = 0$ and so, it has to vanish for all $t \geq 0 $. This proves uniform bounds on $\{\ueps\}_{\varepsilon \in (0,1)}$ and $\{\veps\}_{\varepsilon \in (0,1)}$ in $L^{\infty}((0,\infty)\times\Omega)$ and concludes the proof of global existence. \\

\iffalse
\noindent \color{red} NB: I think we need to assume additionally that $F$ is growing at infinity, i.e. there is $U>0$ and $c>0$ such that $F'(u) > c>0$ for $u> U$. For instance, if $F \in L^{\infty}$ with norm $F_{max}$ and $\|v_0\|_{\infty} > F_{\max}$, then energy equality above does not tell us anything about $\ueps$ (it only tells us that $F((0,\ueps)) \subset (0, F_{max})$ which is trivial). We should at least know that $F$ is strictly increasing a little bit above value $\|v_0\|_{\infty}$. \\
\textcolor{blue}{Agree, I added an assumption in Ass 2.2}

{\noindent  \color{red}NB2: The value of $M$ above is wrong for $\ueps$. Maybe just write $M$ without specifying its value?\\}
\textcolor{blue}{Correct now ?}
\fi

\noindent Estimates~\eqref{PUB1} and \eqref{PUB3} follow directly from \eqref{eq:PDE_for_kinetic_function_withTF} with $\phi(\lambda) = \lambda$. Finally, the estimate on $\Delta \veps$ is deduced by multiplying the equation for $\veps$ by $\varepsilon\,\Delta \veps$.
\end{proof}

\begin{cor}\label{cor:converg_A_v}
Let $\ueps, \veps$ be the solution of system \eqref{system_1}--\eqref{system_2}. Then, $F(\ueps) - \veps \to 0$ strongly in $L^2((0,\infty)\times\Omega)$.
\end{cor}
\noindent Recall that we write $\{\mu_{t,x}\}_{t,x}$ and $\{\nu_{t,x}\}_{t,x}$ for Young measures generated by sequences $\{\ueps\}_{\varepsilon \in (0,1)}$ and $\{\veps\}_{\varepsilon \in (0,1)}$ respectively. We make an elementary observation.
\begin{lem}\label{lem:another_char_of_nu}
Sequence $\{F(\ueps)\}_{\varepsilon \in (0,1)}$ generates Young measure $\left\{F^{\#}\mu_{t,x}\right\}_{t,x}$ (i.e. push-forward of $\mu_{t,x}$ along map $F$). Moreover, for a.e. $(t,x)\in(0,\infty)\times\Omega$ we have
$$
F^{\#}\mu_{t,x}  = \nu_{t,x}.
$$
\end{lem}
\begin{proof}
Let $\Psi: \R \to \R$ be a bounded function and let $\left\{\rho_{t,x} \right\}_{t,x}$ be the Young measure generated by sequence $\{F(\ueps)\}_{\varepsilon \in (0,1)}$. Then, up to a subsequence, weak$^*$ limit of $\Psi(F(\ueps))$ can be written as
$$
\int_{\R} \Psi(\lambda) \diff \rho_{t,x}(\lambda) = \int_{\R} \Psi(F(\lambda)) \diff \mu_{t,x}(\lambda).
$$
Therefore, $\rho_{t,x}=F^{\#} \mu_{t,x}$ holds for a.e. $(t,x)$ as desired. Moreover, Corollary \ref{cor:converg_A_v} shows that $F(\ueps) - \veps \to 0$ strongly in $L^2((0,T)\times\Omega)$. Hence, Young measures generated by these sequences coincide \cite[Lemma 6.3]{MR1452107} and the proof is concluded.
\end{proof}

%-----------------------------
\section{Kinetic formulation}
\label{sect:main_kin_form}
%-----------------------------

\subsection{Kinetic functions and the kinetic PDE}\label{sect:intro_to_knetic_func}

To understand the behaviour of sequences $\{\ueps\}_{\varepsilon \in (0,1)}$ and $\{\veps\}_{\varepsilon \in (0,1)}$, we introduce kinetic function for $\alpha \geq 0$, 
\begin{equation}\label{eq_general_kinetic_function}
\chi_{\alpha}(\xi) = 
\mathds{1}_{0<\xi\leq \alpha}.
\end{equation}
 As Young measures, it is a way to represent nonlinear functions $\varphi: \R \to \R$ since we have a fundamental identity
\begin{equation}\label{sect3:fundament_id_kinetic_form}
\int_{\R} \chi_{\alpha}(\xi) \, \varphi'(\xi) \diff \xi = \int_0^{\alpha} \varphi'(\xi)\diff \xi=\varphi(\alpha)-\varphi(0).
\end{equation}
We let
\begin{equation}\label{eq:pq}
p^{\varepsilon}(t,x,\xi) = \chi_{\ueps(t,x)}(\xi)  \qquad \qquad q^{\varepsilon}(t,x,\xi) = \chi_{\veps(t,x)}(\xi).
\end{equation}
so for any differentiable and bounded $\Psi: \R \to \R$ we have by \eqref{sect3:fundament_id_kinetic_form}
\begin{equation}\label{eq:fund_id_kin_form_part_sequences}
\Psi(\ueps(t,x)) = \Psi(0) + \int_{\R} p^{\varepsilon}(t,x,\xi)\, \Psi'(\xi) \diff \xi, \qquad \Psi(\veps(t,x)) = \Psi(0) + \int_{\R} q^{\varepsilon}(t,x,\xi)\, \Psi'(\xi) \diff \xi. 
\end{equation}
After extraction of a weakly$^*$ converging subsequence in $L^{\infty}((0,T)\times\Omega\times\R)$, we may assume that $p^{\varepsilon} \weaks p$ and $q^{\varepsilon} \weaks q$, i.e.,
\begin{equation}\label{eq:def_of_pq_weak_lim}
p(t,x,\xi) = \wstlim_{\varepsilon\to 0}\, p^{\varepsilon}(t,x,\xi), \qquad \qquad q(t,x,\xi) = \wstlim_{\varepsilon\to 0}\, q^{\varepsilon}(t,x,\xi).
\end{equation}
Connection between $p$ and $q$ will be explored in Lemma ~\ref{lem:prop_k_and_l}. We usually say that kinetic functions $p$ and $q$ are generated by sequences $\{\ueps\}_{\varepsilon \in (0,1)}$ and $\{\veps\}_{\varepsilon \in (0,1)}$ respectively. Basic properties of $p$ and $q$ are recorded below.

\begin{lem}[Properties of $p$ and $q$]\label{lem:prop_k_and_l}
Let $p$ and $q$ be given by \eqref{eq:def_of_pq_weak_lim}. Then, for a.e. $(t,x)\in(0,T)\times\Omega$, 
\begin{enumerate}
\item\label{X1} 
%the maps
we have $0 \leq p(t,x,\xi)\leq 1$ and $0\leq q(t,x,\xi)\leq 1$; 
% are nonnegative functions vanishing for $\xi \leq 0$,
\item\label{X2} the maps $p(t,x,\xi)$ and $q(t,x,\xi)$ are 
%compactly 
supported in $(0,M)$ with $M$ defined in Theorem \ref{theorem:well_posed}, \eqref{PUB0};
%with the size of support independent of $(t,x)$.
\item\label{X3} the maps $\xi \mapsto p(t,x,\xi)$ and $\xi \mapsto q(t,x,\xi)$ are non-increasing,
\item\label{eq:kandlequation} we have
$$q(t,x,\xi) = \int_{\R} p(t,x,\lambda)\,F'(\lambda)\, \delta_{\xi = F(\lambda)}(\xi) \diff \lambda$$ in the sense of distributions.
\end{enumerate}
\end{lem}
\begin{proof}
Property \eqref{X1} follows from the fact that the sequences are nonnegative and this property is preserved under weak limits. Property \eqref{X2} follows from uniform boundedness of sequences $\{\ueps\}_{\varepsilon \in (0,1)}$ and $\{\veps\}_{\varepsilon \in (0,1)}$. Property (3) is a consequence of the same for  $p^{\varepsilon}$ and $q^{\varepsilon}$. To see \eqref{eq:kandlequation}, we fix a smooth test function $\psi(\xi)$ with $\Psi(\xi)=\int_0^{\xi} \psi(\eta) \diff \eta$ and we consider the map 
$$
\R \ni w \mapsto \int_{\R} \psi(\lambda) \, \chi_{F(w)}(\lambda) \diff \lambda = \Psi(F(w)).
$$
With the change of variable $\xi=F(\lambda)$, \eqref{sect3:fundament_id_kinetic_form} implies the identity
$$
\int_{\R} \psi(\xi) \, \chi_{F(w)}(\xi) \diff \xi =
\int_{\R} \psi(F(\lambda)) \, \chi_{w}(\lambda) \,  F'(\lambda) \diff \lambda = \int_{\R} \psi(\xi) \int_{\R} \chi_{w}(\lambda) \, F'(\lambda) \, \delta_{\xi=F(\lambda)}(\xi) \diff \lambda \diff \xi.
$$
We plug $w = \ueps(t,x)$ and deduce
\begin{equation}\label{eq:kandl_almost_pass}
\int_{\R} \psi(\xi) \, \chi_{F(\ueps)}(\xi) \diff \xi = \int_{\R} \psi(\xi) \int_{\R} p^{\varepsilon}(t,x,\xi) \, F'(\lambda) \, \delta_{\xi=F(\lambda)}(\xi) \diff \lambda \diff \xi.
\end{equation}
Now, to identify the weak$^*$ limit on the (LHS) of \eqref{eq:kandl_almost_pass}, we note that
$$
\left|\int_{\R} \psi(\xi) \, \chi_{F(\ueps)}(\xi) \diff \xi  - \int_{\R} \psi(\xi) \, \chi_{\veps}(\xi) \diff \xi  \right| = \left|\Psi(F(\ueps)) - \Psi(\veps) \right| \to 0 \mbox{ in } L^2((0,T)\times\Omega)
$$
due to Corollary \ref{cor:converg_A_v}. Therefore, sending $\varepsilon \to 0$ in \eqref{eq:kandl_almost_pass}, we obtain \eqref{eq:kandlequation}.
\end{proof}
\noindent In view of Lemma \ref{lem:another_char_of_nu}, let us also connect kinetic functions with Young measures.
\begin{lem}[Young measures vs kinetic formulation]\label{lem:YM_con_with_kf}
Let $p$ and $q$ be given by \eqref{eq:def_of_pq_weak_lim} and let $\{\mu_{t,x}\}_{t,x}$ and $\{\nu_{t,x}\}_{t,x}$ be the Young measures generated by $\{\ueps\}_{\varepsilon \in (0,1)}$ and $\{\veps\}_{\varepsilon \in (0,1)}$ respectively. Then, for a.e. $(t,x)\in(0,T)\times\Omega$ we have, in the sense of distributions, 
$$
{\partial_\xi}\, p(t,x,\xi) = \delta_0 - \mu_{t,x}, \qquad \qquad {\partial_\xi}\, q(t,x,\xi) = \delta_0 - \nu_{t,x} = \delta_0 - F^{\#}\mu_{t,x} .
$$

\end{lem}
\begin{proof}
We only prove the first formula as the second follows from a similar reasoning combined with Lemma \ref{lem:another_char_of_nu}. Let $\Psi:\R \to \R$ be an arbitrary smooth and bounded function. Passing to the limit $\varepsilon \to 0$ in \eqref{eq:fund_id_kin_form_part_sequences} we obtain
$$
\int_{\R} \Psi(\lambda)\diff\mu_{t,x}(\lambda) = \Psi(0) + \int_{\R} p(t,x,\xi)\,\Psi'(\xi) \diff \xi
$$
which proves the first identity. 
%The second one follows from a similar argument combined with Lemma \ref{lem:another_char_of_nu}.
\end{proof}

\noindent We conclude this subsection with a distributional PDE that will be exploited in the compactness result (Lemma \ref{lem:compactness_product}).
\begin{thm}[PDE satisfied by kinetic functions]\label{thm:2ndPDE_kinfunct}
Let $p$ and $q$ be given by \eqref{eq:def_of_pq_weak_lim}. Then, there is a uniformly bounded sequence of nonnegative measures $\{n^{\varepsilon}\}_{\varepsilon \in (0,1)}$ on $(0,T)\times\Omega\times\R$ such that
\begin{equation}\label{eq:PDE2_sat_by_kin_seq}
\partial_t \left[\int_{\R} p^{\varepsilon}(t,x,\lambda) \, \delta_{\xi = F(\lambda)}(\xi) \diff \lambda +  q^{\varepsilon}(t,x,\xi)\right] - \Delta_x q^{\varepsilon}(t,x,\xi) =  \partial_{\xi}\, {n}^{\varepsilon}(t,x,\xi)
\end{equation}
in the sense of distributions. In particular, there is a bounded nonnegative measure $n$ on $(0,T)\times\Omega\times\R$ such that 
\begin{equation}\label{eq:PDE2_sat_by_kin_LIMIT}
\partial_t \left[\int_{\R} p(t,x,\lambda) \, \delta_{\xi = F(\lambda)}(\xi) \diff \lambda +  q(t,x,\xi)\right] - \Delta_x q(t,x,\xi) =  \partial_{\xi}\, {n}(t,x,\xi).
\end{equation}
\end{thm}
\begin{proof}
We consider smooth test function $\phi: \R \to \R$ as well as $\Psi$ and $\Phi$ defined by \eqref{eq:def_of_Psi_Phi}. From \eqref{eq:PDE_for_kinetic_function_withTF} we know that
$$
\partial_t \Psi(\ueps) + \partial_t \Phi(\veps) = \Delta \Phi(\veps) - \phi'(\veps)\,|\nabla \veps|^2 - \frac{\big(\veps - F(\ueps)\big)\, (\phi(\veps) - \phi(F(\ueps))\big)}{\varepsilon}.
$$
Now, using kinetic functions we can write $\partial_t \Psi(\ueps)$ as
$$
\partial_t \Psi(\ueps) = \partial_t \int_{\R} p^{\varepsilon}(t,x,\lambda) \, \phi(F(\lambda)) \diff \lambda = \partial_t  \int_{\R} \phi(\xi) \, \int_{\R}  p^{\varepsilon}(t,x,\lambda) \,  \delta_{\xi = F(\lambda)}(\xi) \diff \xi \diff \lambda.
$$
while $\partial_t\Phi(\veps)$ and $\Delta \Phi(\veps)$ as
$$
\partial_t \Phi(\veps) = \partial_t \int_{\R} \phi(\xi) \, q^{\varepsilon}(t,x,\xi) \diff \xi, \qquad \Delta \Phi(\veps) = \Delta \int_{\R} \phi(\xi) \, q^{\varepsilon}(t,x,\xi) \diff \xi.
$$
Now, term $-\, \phi'(\veps) \, |\nabla \veps|^2$ can be interpreted as a derivative of a nonnegative measure $n_1^{\varepsilon}(t,x,\xi)$:
$$
-\, \phi'(\veps) \, |\nabla \veps|^2 = \left\langle \phi(\xi), {\partial_\xi} \, \delta_{\xi = \veps(t,x)}(\xi)\, |\nabla \veps(t,x)|^2 \right\rangle := \Big\langle \phi(\xi),\, {\partial_{\xi}} \, n_1^{\varepsilon}(t,x,\xi) \Big\rangle,
$$
We note that the sequence $\left\{n_1^{\varepsilon}\right\}_{\varepsilon \in (0,1)}$ is uniformly bounded in the space of measures according to estimate \eqref{PUB1} in Theorem \ref{theorem:well_posed}. Similarly, using
\begin{multline*}
\phi(\veps) - \phi(F(\ueps)) = \int_0^1 \frac{\diff}{\diff s} \phi(s\, \veps + (1-s) \, F(\ueps)) \diff s = \\ =(\veps - F(\ueps)) \, \int_0^1 \phi'(s\, \veps + (1-s) \, F(\ueps)) \diff s,
\end{multline*}
we can write
$$
-\, \frac{\big(\veps - F(\ueps)\big)\, \big(\phi(\veps) - \phi(F(\ueps))\big)}{\varepsilon} = 
\frac{\big(\veps - F(\ueps)\big)^2}{\varepsilon} \, \left\langle \phi(\xi), \partial_{\xi} \int_0^1 \delta_{s\, \veps + (1-s) \, F(\ueps)}(\xi) \diff s \right\rangle
$$
where the last term can be interpreted as a Bochner integral in the space of measures. We let
\begin{equation}\label{eq:term_n2_measure}
n_2^{\varepsilon}(t,x,\xi) = \frac{\big(\veps - F(\ueps)\big)^2}{\varepsilon} \,\int_0^1 \delta_{s\, \veps + (1-s) \, F(\ueps)}(\xi) \diff s
\end{equation}
which is a uniformly bounded sequence of nonnegative measures on $(0,T)\times\Omega\times\R$ due to \eqref{PUB3} in Theorem \ref{theorem:well_posed}. Therefore, it has a subsequence converging weakly$^*$, i.e. $n_2^{\varepsilon} \weaks n_2$. Collecting all the terms, we obtain distributional identity \eqref{eq:PDE2_sat_by_kin_seq}. Passing to the weak$^*$ limit with $\varepsilon \to 0$, we deduce~\eqref{eq:PDE2_sat_by_kin_LIMIT} and Theorem~\ref{thm:2ndPDE_kinfunct} is proved.
\end{proof}

%------------------------------------------------------
\subsection{A kinetic identity satisfied by $p$ and $q$} 
%------------------------------------------------------

Now, following \cite{MR1299852} for Young measures, we formulate a distributional identity that will be used to identify kinetic functions $p$ and $q$. We start with the result of compensated compactness type.
\begin{lem}[compensated compactness]\label{lem:compactness_product}
Let $p$ and $q$ be given by \eqref{eq:def_of_pq_weak_lim}. Then,
$$
q^{\varepsilon}(t,x,\xi)\,\left[\int_{\R} 
p^{\varepsilon}(t,x,\lambda) \, \delta_{\eta = F(\lambda)}(\eta) \diff \lambda + q^{\varepsilon}(t,x,\eta)\right] \to
q(t,x,\xi)\,\left[\int_{\R} 
p(t,x,\lambda) \, \delta_{\eta = F(\lambda)}(\eta) \diff \lambda + q(t,x,\eta)\right]
$$
in the sense of distributions. More precisely, for all smooth and compactly supported test functions $\phi,\,\psi:\R \to \R$ and $\varphi:(0,T)\times\Omega \to \R$ we have
\begin{multline*}
\int_{(0,T)\times \Omega} \varphi(t,x) \, \int_{\R} \phi(\xi)\, q^{\varepsilon}(t,x,\xi) \diff \xi \, \int_{\R}  \psi(\eta)\,\left[ \int_{\R} p^{\varepsilon}(t,x,\tau)\, \delta_{\eta = F(\tau)}(\eta)  \diff \tau + q^{\varepsilon}(t,x,\eta)\right]\diff \eta \diff x \diff t \to \\
\to \int_{(0,T)\times \Omega} \varphi(t,x) \, \int_{\R} \phi(\xi)\, q(t,x,\xi) \diff \xi \, \int_{\R}  \psi(\eta)\,\left[\int_{\R} p(t,x,\tau)\, \delta_{\eta = F(\tau)}(\eta)  \diff \tau + q(t,x,\eta)\right]\diff \eta \diff x \diff t.
\end{multline*}
\end{lem}
\begin{proof}
Let
$$
\mathcal{P}^{\varepsilon}_{\psi}(t,x) = \int_{\R}  \psi(\eta)\,\left[\int_{\R} p^{\varepsilon}(t,x,\xi)(t,x,\tau)\, \delta_{\eta = F(\tau)}(\eta)  \diff \tau + q(t,x,\eta)\right]\diff \eta, $$
$$\mathcal{Q}^{\varepsilon}_{\phi}(t,x) =
\int_{\R} \phi(\xi)\, q^{\varepsilon}(t,x,\xi) \diff \xi.
$$
Setting $\mathcal{P}^{\varepsilon}_{\psi} \weaks \mathcal{P}_{\psi}$ and $\mathcal{Q}^{\varepsilon}_{\phi} \weaks \mathcal{Q}_{\phi}$, we have to prove that
$$
\int_{(0,T)\times\Omega} \varphi(t,x) \, \mathcal{P}^{\varepsilon}_{\psi}(t,x) \, \mathcal{Q}^{\varepsilon}_{\phi}(t,x) \diff t \diff x \to \int_{(0,T)\times\Omega} \varphi(t,x) \, \mathcal{P}_{\psi}(t,x) \, \mathcal{Q}_{\phi}(t,x) \diff t \diff x.
$$
Consider operator $T: L^2(\Omega) \to H^2(\Omega)$ defined as the solution of the Neumann problem
\begin{equation}\label{eq:auxillary_problem_laplacian}
- \Delta \left[T(f)\right] = \mu \, \left[T(f)\right] + f \mbox{ in } \Omega, \qquad \qquad \frac{\partial}{\partial {\bf \mbox{n}}} \left[T(f)\right] = 0 \mbox{ on } \partial \Omega,
\end{equation}
where $\mu > 0$ is some fixed number from the resolvent of $-\Delta$ with Neumann boundary conditions.
%\cite{MR3124880}. 
By elliptic regularity theory, $\|T(f)\|_{H^2(\Omega)} \leq C\, \|f \|_{L^2(\Omega)}$ for a.e. $t \in (0,T)$. In particular, up to a subsequence, uniqueness of solutions to \eqref{eq:auxillary_problem_laplacian} implies  $T\left(\mathcal{P}^{\varepsilon}_{\psi}\right) \weak T\left(\mathcal{P}^{\phantom{\varepsilon}}_{\psi}\right)$ in $L^2(0,T;H^1(\Omega))$.
Using the operator $T$ we can write
$$
\int_{(0,T)\times\Omega} \varphi \, \mathcal{P}^{\varepsilon}_{\psi} \, \mathcal{Q}^{\varepsilon}_{\phi} \diff t \diff x = \int_{(0,T)\times\Omega} \nabla T\left(\mathcal{P}^{\varepsilon}_{\psi}\right) \cdot \nabla\left( \varphi\, \mathcal{Q}^{\varepsilon}_{\phi} \right) \diff t \diff x - \mu \, \int_{(0,T)\times\Omega} T\left(\mathcal{P}^{\varepsilon}_{\psi}\right) \, \left(\varphi\,\mathcal{Q}^{\varepsilon}_{\phi}\right) \diff t \diff x .
$$
Clearly, up to a subsequence, $\mathcal{Q}^{\varepsilon}_{\phi} \weak \mathcal{Q}_{\phi}$ in $L^2(0,T;H^1(\Omega))$ because $\mathcal{Q}^{\varepsilon}_{\phi} = \Phi(\veps)$ (where $\Phi' = \phi$) and $\{\veps\}_{\varepsilon \in (0,1)}$ is bounded in $L^2(0,T;H^1(\Omega))$ cf. \eqref{PUB1} in Theorem \ref{theorem:well_posed}. Therefore, it is sufficient to prove that $T\left(\mathcal{P}^{\varepsilon}_{\phi}\right) \to T\left(\mathcal{P}_{\phi}^{\phantom{\varepsilon}}\right)$ strongly in $L^2(0,T;H^1(\Omega))$.\\

\noindent We want to apply Aubin-Lions Lemma for the case where time derivative is a measure cf. \cite[Corollary 7.9]{MR3014456}.
%or \cite[Section 4]{MR1857155}. 
In view of the regularity estimate $$
\left\|T\left(\mathcal{P}^{\varepsilon}_{\psi}\right)(t,x)\right\|_{L^2(0,T;H^2(\Omega))} \leq C\, \left\|\mathcal{P}^{\varepsilon}_{\psi}(t,x) \right\|_{L^2(0,T;L^2(\Omega))},
$$
we only have to prove that the sequences of distributional time derivatives $\left\{\frac{\partial}{\partial t}T\left(\mathcal{P}^{\varepsilon}_{\psi} \right)\right\}_{\varepsilon \in (0,1)}$ and $\left\{\frac{\partial}{\partial t}\nabla T\left(\mathcal{P}^{\varepsilon}_{\psi} \right) \right\}_{\varepsilon \in (0,1)}$ are bounded in $C(0,T;X)^*$ for some separable Banach space $X$ such that $L^2(\Omega) \subset X^*$.\\

\noindent To this end, we note that equation \eqref{eq:PDE2_sat_by_kin_seq} implies that sequence $\left\{\frac{\partial}{\partial t}\mathcal{P}^{\varepsilon}_{\psi} \right\}_{\varepsilon \in (0,1)}$ is bounded in space $\left[C(0,T; H^k(\Omega))\right]^*$ for some $k > d$ so that $H^k(\Omega)$ embedds continuously into $L^{\infty}(\Omega)$ cf. \cite[Corollary 7.11]{MR1814364}. Then, we claim that equation \eqref{eq:auxillary_problem_laplacian} implies that
$$
\left\{\frac{\partial}{\partial t}T\left(\mathcal{P}^{\varepsilon}_{\psi}\right) \right\}_{\varepsilon \in (0,1)} \mbox{ is bounded in } \left[C(0,T; H^{k-2}(\Omega))\right]^*,
$$
$$
\left\{\frac{\partial}{\partial t}\nabla T\left(\mathcal{P}^{\varepsilon}_{\psi}\right) \right\}_{\varepsilon \in (0,1)} \mbox{ is bounded in } \left[C(0,T; H^{k-1}(\Omega))\right]^*.
$$
Indeed, if $\zeta(t,x)$ is a (vector-valued) smooth and compactly supported test function we have
\begin{align*}
&\int_{(0,T)\times\Omega} \nabla T\left(\mathcal{P}^{\varepsilon}_{\psi}\right) (t,x) \cdot \partial_t \zeta(t,x) \diff t \diff x = \\
& \qquad \qquad =\,- \int_{(0,T)\times\Omega} \nabla T\left(\mathcal{P}^{\varepsilon}_{\psi}\right) (t,x) \cdot \Big[\Delta T\left(\partial_t \zeta\right)(t,x) + \mu \, T\left(\partial_t \zeta\right)(t,x) \Big] \diff t \diff x \, \\
& \qquad \qquad = - \int_{(0,T)\times\Omega} \mathcal{P}^{\varepsilon}_{\psi} \, \partial_t \left(\DIV T(\zeta)\right) \diff t \diff x \leq C \, \| \DIV T(\zeta) \|_{C(0,T;H^k(\Omega))} \\  & \qquad \qquad \leq
C\, \| T(\zeta) \|_{C(0,T;H^{k+1}(\Omega))} \leq
C \, \| \zeta \|_{C(0,T;H^{k-1}(\Omega))}. \phantom{\int_{(0,T)\times\Omega}}
\end{align*}
Similar computation can be performed for the term $ \int_{(0,T)\times\Omega} T\left(\mathcal{P}^{\varepsilon}_{\psi}\right) (t,x) \, \partial_t \zeta(t,x) \diff t \diff x
$. This concludes the proof.
\end{proof}
\noindent We are in position to formulate a functional identity relating the kinetic functions $p$ and $q$.

\begin{thm}[Functional identity for $p$ and $q$]\label{thm:funct_id_forkandl}
Let $p$ and $q$ be given by \eqref{eq:def_of_pq_weak_lim}. Then, the following identity is satisfied in the sense of distributions
\begin{equation}\label{thmeq:main_identity}
\begin{split}
q(t,x,\xi)\,\int_{\R} 
p(t,x,\lambda) \, \delta_{\eta = F(\lambda)}(\eta)& \diff \lambda = \chi_{\eta}(\xi) \, \int_{\R}  p(t,x,\lambda) \,   \delta_{\eta = F(\lambda)}(\eta)  \diff \lambda
\\
& \phantom{=} + \chi_{\xi}(\eta)  \,q(t,x,\xi) + \chi_{\eta}(\xi) \, q(t,x,\eta) - q(t,x,\eta)\, q(t,x,\xi)  \phantom{\int_{\R}}\\
& \phantom{=} + \int_{\R} \int_{\R}    p(t,x,\lambda) \, \chi_{\lambda}(\tau) \, \delta_{\eta = F(\tau)}(\eta)  \,  \delta_{\xi = F(\lambda)}(\xi)  \, F'(\lambda) \diff \lambda \diff \tau.
\end{split}
\end{equation}
\end{thm}
\begin{proof}
Consider two smooth test functions $\phi$, $\varphi: \R \to \R$ and define
$$
\Psi(\lambda):= \int_0^{\lambda} \phi(F(\tau))\diff \tau, \qquad \Phi(\lambda):= \int_0^{\lambda} \phi(\tau) \diff \tau, \qquad \Theta(\lambda) := \int_0^{\lambda} \varphi(\tau) \diff \tau.
$$
Now, the plan is to consider the limit as $\varepsilon \to 0$ of the expression
$$
\left[\Psi(\ueps) + \Phi(\veps)\right]\,\Theta(\veps).
$$
From the kinetic representation~\eqref{sect3:fundament_id_kinetic_form},  we can write
$$
\Psi(\ueps) = \int_{\R} \phi(F(\lambda)) \, p^{\varepsilon}(t,x,\lambda) \diff \lambda = \int_{\R} \int_{\R} \phi(\eta) \, \delta_{\eta = F(\tau)}(\eta) \, p^{\varepsilon}(t,x,\lambda) \diff \eta \diff \lambda,
$$
$$
\Phi(\veps) = \int_{\R} \phi(\eta) \, q^{\varepsilon}(t,x,\eta) \diff \eta, \qquad \qquad  \Theta(\veps) = \int_{\R} \varphi(\xi) \, q^{\varepsilon}(t,x,\xi) \diff \xi.
$$
Onthe one hand, using Lemma \ref{lem:compactness_product}, we obtain
$$
\left[\Psi(\ueps) + \Phi(\veps)\right]\,\Theta(\veps) \weaks \int_{\R} \int_{\R} \phi(\eta)\, \varphi(\xi)\, q(t,x,\xi)\,\left[\int_{\R} 
p(t,x,\lambda) \, \delta_{\eta = F(\lambda)}(\eta) \diff \lambda + q(t,x,\eta)\right] \diff \eta \diff \xi.
$$
On the other hand, we can replace the term $\left[\Psi(\ueps) + \Phi(\veps)\right]\,\Theta(\veps)$ with $\left[\Psi(\ueps) + \Phi(F(\ueps))\right]\,\Theta(F(\ueps))$ because $\veps - F(\ueps) \to 0$ strongly. Therefore, we can use the kinetic representation
$$
\left[\Psi(\ueps) + \Phi(F(\ueps))\right]\,\Theta(F(\ueps)) = \int_{\R} p^{\varepsilon}(t,x,\lambda) \big[\left(\Psi(\lambda) + \Phi(F(\lambda))\right) \, \Theta(F(\lambda)) \big]' \diff \lambda =: A^{\varepsilon} + B^{\varepsilon}+ C^{\varepsilon}
$$
where these three terms come from the differentiation of the product. We have
\begin{align*}
A^{\varepsilon} &= \int_{\R} p^{\varepsilon}(t,x,\lambda) \, \Theta(F(\lambda)) \, \left(\Psi'(\lambda) + \Phi'(F(\lambda)) \, F'(\lambda) \right) \diff \lambda \\
&= \int_{\R} p^{\varepsilon}(t,x,\lambda) \, \left[\int_{\R} \chi_{F(\lambda)}(\xi) \, \varphi(\xi) \diff \xi \right] \, \left(\phi(F(\lambda)) + \phi(F(\lambda)) \, F'(\lambda) \right) \diff \lambda \\
&= \int_{\R} p^{\varepsilon}(t,x,\lambda) \, \left[\int_{\R} \chi_{F(\lambda)}(\xi) \, \varphi(\xi) \diff \xi \right] \, \left[\int_{\R} \phi(\eta) \, \delta_{\eta = F(\lambda)}(\eta) \diff \lambda \right]\, \left(1 + F'(\lambda) \right) \diff \lambda \\
&= \int_{\R} \int_{\R} \int_{\R} \varphi(\xi)  \,  \phi(\eta) \, p^{\varepsilon}(t,x,\lambda) \, \chi_{\eta}(\xi) \,  \delta_{\eta = F(\lambda)}(\eta) \, \left(1 + F'(\lambda) \right) \diff \lambda \diff \eta \diff \xi,
\end{align*}
\begin{align*}
B^{\varepsilon} &= \int_{\R} p^{\varepsilon}(t,x,\lambda) \,     \Psi(\lambda)  \, \Theta'(F(\lambda)) \, F'(\lambda) \diff \lambda \\
&= \int_{\R} p^{\varepsilon}(t,x,\lambda) \,      \left[\int_{\R} \int_{\R} \chi_{\lambda}(\tau) \, \phi(\eta) \, \delta_{\eta = F(\tau)}(\eta) \diff \tau \diff \eta \right] \, \varphi(F(\lambda)) \, F'(\lambda) \diff \lambda \\
&= \int_{\R} p^{\varepsilon}(t,x,\lambda) \,      \left[\int_{\R} \int_{\R} \chi_{\lambda}(\tau) \, \phi(\eta) \, \delta_{\eta = F(\tau)}(\eta) \diff \tau \diff \eta \right] \, \left[\int_{\R} \varphi(\xi) \, \delta_{\xi = F(\lambda)}(\xi) \diff \xi \right] \, F'(\lambda) \diff \lambda\\
&= \int_{\R}\int_{\R}\int_{\R} \int_{\R}   \phi(\eta) \, \varphi(\xi) \, p^{\varepsilon}(t,x,\lambda) \, \chi_{\lambda}(\tau) \, \delta_{\eta = F(\tau)}(\eta)  \,  \delta_{\xi = F(\lambda)}(\xi)  \, F'(\lambda) \diff \lambda \diff \tau \diff \xi \diff \eta 
\end{align*}
and
\begin{align*}
C^{\varepsilon} &= \int_{\R} p^{\varepsilon}(t,x,\lambda) \, \Phi(F(\lambda)) \, \Theta'(F(\lambda)) \, F'(\lambda) \diff \lambda \\
&= \int_{\R} p^{\varepsilon}(t,x,\lambda) \, \left[\int_{\R} \phi(\eta) \, \chi_{F(\lambda)}(\eta) \diff \eta \right] \, \left[\int_{\R} \varphi(\xi) \, \delta_{\xi = F(\lambda)}(\xi) \diff \xi \right] \, F'(\lambda) \diff \lambda \\
&= \int_{\R}\int_{\R} \int_{\R}\phi(\eta) \,\varphi(\xi) \, p^{\varepsilon}(t,x,\lambda) \,  \chi_{\xi}(\eta)  \,  \delta_{\xi = F(\lambda)}(\xi)   \, F'(\lambda) \diff \lambda \diff \xi \diff \eta.
\end{align*}
Passing to the weak$^*$ limit $\varepsilon \to 0$ in these three terms, we obtain
\begin{equation}\label{eq:identity_pq_before_simpl}
\begin{split}
&q(t,x,\xi)\,\int_{\R} 
p(t,x,\lambda) \, \delta_{\eta = F(\lambda)}(\eta) \diff \lambda + q(t,x,\eta)\, q(t,x,\xi) =\\&\qquad = \chi_{\xi}(\eta)  \, \int_{\R} p(t,x,\lambda) \, \delta_{\xi = F(\lambda)}(\xi)   \, F'(\lambda) \diff \lambda + \chi_{\eta}(\xi) \, \int_{\R}  p(t,x,\lambda) \,   \delta_{\eta = F(\lambda)}(\eta) \, \left(1 + F'(\lambda) \right) \diff \lambda\\
& \qquad \phantom{=} + \int_{\R} \int_{\R}    p(t,x,\lambda) \, \chi_{\lambda}(\tau) \, \delta_{\eta = F(\tau)}(\eta)  \,  \delta_{\xi = F(\lambda)}(\xi)  \, F'(\lambda) \diff \lambda \diff \tau
\end{split}
\end{equation}
understood in the sense of distributions. Using \eqref{eq:kandlequation} in Lemma~\ref{lem:prop_k_and_l}
we simplify \eqref{eq:identity_pq_before_simpl} to \eqref{thmeq:main_identity}. 
\end{proof}

%---------------------------------
\section{Proof of the Young measure representation}
\label{sect:prf_main_result}
%----------------------------------

\noindent We are in position to prove the representation of $u$ by Young measures as stated in Theorem \ref{thm_main_id_YM}. It turns out, that equations \eqref{eq:kandlequation} in Lemma~\ref{lem:prop_k_and_l} and \eqref{thmeq:main_identity} completely characterize the kinetic function $q$ (and $p$). The proof is based on the identity \eqref{thmeq:main_identity} with fixed $(t,x) \in (0,T)\times\Omega$ so we simplify notations.
\begin{Nott}
In this section, we fix $(t,x) \in (0,T)\times\Omega$ and write $p(\xi)$ and $q(\xi)$ for $p(t,x,\xi)$ and $q(t,x,\xi)$ respectively.
\end{Nott}
\noindent In order to translate distributional identity \eqref{thmeq:main_identity} into a functional one, we use the following 
\begin{lem}[Adjoint distribution to $\delta_{\xi = F(\tau)}(\xi)$]\label{lem:int_by_subst_part1}
There exists a distribution 
$$
\delta_{\xi = F(\tau)}^*(\tau):=\sum_{i=1}^3 \delta_{\tau=S_i(\xi)}(\tau)\, \left|S_i'(\xi)\right|
$$
with  $S_1$, $S_2$, $S_3$ defined in Notation \ref{intro:not_inv_phi}, such that, for all test functions $\psi$, $\Theta$, 
$$
\int_{\R} \Theta(\tau) \int_{\R} \Psi(\xi)\, \delta_{\xi = F(\tau)}(\xi) \diff \xi \diff \tau = \int_{\R} \Psi(\xi) \int_{\R} \Theta(\tau)\,\delta_{\xi = F(\tau)}^*(\tau) \diff \tau \diff \xi.
$$
%It is given by the formula, with  $S_1$, $S_2$, $S_3$ defined in Notation \ref{intro:not_inv_phi},
%$$
% \int_{\R} \Theta(\tau)\,\delta_{\xi = F(\tau)}^*(\tau) \diff \tau = \sum_{i=1}^3 %\Theta(S_i(\xi))\, \left|S_i'(\xi)\right| .$$
\end{lem}
\begin{proof}
By definition
$$
\int_{\R} \Theta(\tau) \int_{\R} \Psi(\xi)\, \delta_{\xi = F(\tau)}(\xi) \diff \xi \diff \tau = \int_{\R} \Theta(\tau) \, \Psi(F(\tau)) \diff \tau
$$
Let $\{I_i\}_{i=1}^3$ and $\{J_i\}_{i=1}^3$ be as in Notation \ref{intro:not_inv_phi}. Then, we can integrate by substitution
$$
\int_{\R} \Theta(\tau) \, \Psi(F(\tau)) \diff \tau = 
\sum_{i=1}^3 \int_{I_i} \Theta(\tau) \, \Psi(F(\tau)) \diff \tau = \sum_{i=1}^3 \int_{J_i} \Theta(S_i(\xi)) \, \Psi(\xi) \, \left|S_i'(\xi) \right| \diff \xi. 
$$
As inverses are extended by a constant to the whole of $\R$ we can write
$$
\int_{\R} \Theta(\tau) \, \Psi(F(\tau)) \diff \tau = \sum_{i=1}^3 \int_{\R} \Theta(S_i(\xi)) \, \Psi(\xi) \, \left|S_i'(\xi)\right| \diff \xi = \int_{\R} \Psi(\xi) \int_{\R} \Theta(\tau)\,\delta_{\xi = F(\tau)}^*(\tau) \diff \tau \diff \xi. 
$$
\end{proof}
\begin{cor}\label{rem:practical_connection_betkandl} 
Identity \eqref{eq:kandlequation} in Lemma~\ref{lem:prop_k_and_l} can be also written explicitly as 
\begin{equation}
q(\xi) =  \sum_{i=1}^3 (-1)^{i+1}\, p(S_i(\xi))\, \mathds{1}_{J_i}(\xi).
\end{equation}
This follows from Lemma \ref{lem:int_by_subst_part1} and an observation that $F'(S_i(\lambda))\,|S_i'(\lambda)| = (-1)^{i+1}\,\mathds{1}_{J_i}(\lambda)$ where $\{J_i\}_{i=3}^3$ are defined in Notation \ref{intro:not_inv_phi}.
\end{cor}
\begin{lem}[Explicit formulation of the kinetic identity]
The functional identity \eqref{thmeq:main_identity} can be written explicitly as
\begin{equation}\label{thmeq:main_identity_loctx}
\left[q(\xi) - \chi_{\eta}(\xi) \right] \, \mathcal{S}(\eta) = \chi_{\xi}(\eta)\,q(\xi) + \chi_{\eta}(\xi)\,q(\eta) - q(\eta)\,q(\xi) +   \mathcal{R}(\eta,\xi),
\end{equation} where 
\begin{equation}\label{easier_formula_for_S_and_R}
\mathcal{S}(\eta) = \sum_{i=1}^3 p(S_i(\eta))\, |S_i'(\eta)|, \quad  \mathcal{R}(\eta,\xi)=\sum_{i=1}^3 (-1)^{i+1} \,
p(S_i(\xi)) \, \mathds{1}_{J_i}(\xi)\, \sum_{j=1}^3 \chi_{S_i(\xi)}(S_j(\eta)) \, |S_j'(\eta)|.
\end{equation}
\end{lem}
\begin{proof}
Using adjoint distribution from Lemma \ref{lem:int_by_subst_part1}, distributional identity \eqref{thmeq:main_identity} can be reformulated as follows. For a.e. $\eta,\, \xi > 0$, it holds
\begin{equation*}
\begin{split}
q(\xi) \, \int_{\R} p(\lambda) \, \delta_{\eta = F(\lambda)}^*(\lambda)   \diff \lambda &=  \, \chi_{\eta}(\xi)\,\int_{\R} 
p(\lambda) \, \delta_{\eta = F(\lambda)}^*(\lambda)
\diff \lambda \,+\\&+  \int_{\R} \int_{\R}  p(\lambda) \,    \chi_{\lambda}(\tau) \, F'(\lambda) \, \delta_{\eta= F(\tau)}^*(\tau) \, \delta_{\xi=F(\lambda)}^*(\lambda) \diff \tau  \diff \lambda\\
&+ \chi_{\xi}(\eta)\,q(\xi) + \chi_{\eta}(\xi)\,q(\eta) - q(\eta)\,q(\xi). \phantom{\int_{\R}}
\end{split}
\end{equation*}
If we define 
$$
\mathcal{S}(\eta) = \int_{\R} p(\lambda) \, \delta_{F(\tau)=\eta}^*(\tau) \diff \tau, \qquad  \mathcal{R}(\eta,\xi)=\int_{\R} \int_{\R}  p(\lambda) \,    \chi_{\lambda}(\tau) \, F'(\lambda) \, \delta_{F(\tau)=\eta}^*(\tau) \, \delta_{F(\lambda)=\xi}^*(\lambda) \diff \tau  \diff \lambda,
$$
we obtain \eqref{thmeq:main_identity_loctx} and we just have to prove the claimed formulas for $\mathcal{S}(\eta)$ and $\mathcal{R}(\eta,\xi)$ as in \eqref{easier_formula_for_S_and_R}. Using Lemma \ref{lem:int_by_subst_part1}, we have
$$
\mathcal{S}(\eta) = \sum_{i=1}^3 p(S_i(\eta))\, |S_i'(\eta)|.
$$
For $\mathcal{R}(\eta,\xi)$ we additionally note that $ F'(S_i(\xi)) \, |S_i'(\xi)| = (-1)^{i+1} \, \mathds{1}_{J_i}(\xi)$. Hence,
\begin{equation*}
\begin{split}
&\mathcal{R}(\eta,\xi) = \sum_{i=1}^3
p(S_i(\xi))  \, F'(S_i(\xi)) \, |S_i'(\xi)|\, \sum_{j=1}^3 \chi_{S_i(\xi)}(S_j(\eta)) \, |S_j'(\eta)| = \\ & \qquad \qquad \qquad \qquad = 
\sum_{i=1}^3 (-1)^{i+1} \,
p(S_i(\xi)) \, \mathds{1}_{J_i}(\xi)\, \sum_{j=1}^3 \chi_{S_i(\xi)}(S_j(\eta)) \, |S_j'(\eta)|.
\end{split}
\end{equation*}
\end{proof}
\noindent Formula for $\mathcal{R}(\eta,\xi)$ seems to be complicated. We compute its value for $\eta$ and $\xi$ in unstable region below.
\begin{lem}[$\mathcal{R}(\eta,\xi)$ in the unstable region]\label{R_unstable_region} Let $\mathcal{R}(\eta,\xi)$ be defined with \eqref{easier_formula_for_S_and_R}. Moreover, let $\xi_1$, $\eta$ and $\xi_2$ be such that $f_{-} < \xi_1 < \eta <  \xi_2 < f_{+}$. Then,
\begin{equation}\label{eq:formulaforRetaxi1}
\mathcal{R}(\eta, \xi_1) = \big[p(S_3(\xi_1)) - p(S_2(\xi_1))  \big] \, \big[S_1'(\eta) - S_2'(\eta)\big],
\end{equation}
\begin{equation}\label{eq:formulaforRetaxi2}
\mathcal{R}(\eta,\xi_2) = q(\xi_2) \, S_1'(\eta) + p(S_3(\xi_2)) \, \left[S_3'(\eta) - S_2'(\eta)\right].
\end{equation}
\end{lem}
\begin{proof}
We compute explicitly $\mathcal{R}(\eta, \xi_1)$ and $\mathcal{R}(\eta,\xi_2)$. As it can be seen from formula \eqref{easier_formula_for_S_and_R}, it is important to understand how $S_i(\xi_1)$ and $S_i(\xi_2)$ are related to $S_j(\eta)$. As $f_{-} < \xi_1 < \eta <  \xi_2 < f_{+}$, we deduce from Fig.~\ref{plot:Fd} that
$$
\alpha_{-} < S_1(\xi_1) < S_1(\eta) < S_2(\eta) < S_2(\xi_1) < S_3(\xi_1) < S_3(\eta) < \beta_{+}.
$$
Therefore,
$$
\sum_{j=1}^3 \chi_{S_1(\xi_1)}(S_j(\eta)) \, |S_j'(\eta)| = 0, \quad  \sum_{j=1}^3 \chi_{S_2(\xi_1)}(S_j(\eta)) \, |S_j'(\eta)| = \sum_{j=1}^3 \chi_{S_3(\xi_1)}(S_j(\eta)) \, |S_j'(\eta)| = S_1'(\eta) - S_2'(\eta).
$$
Hence, \eqref{easier_formula_for_S_and_R} implies
\begin{equation*}
\mathcal{R}(\eta, \xi_1) = \sum_{i=2}^3 (-1)^{i+1} \,
p(S_i(\xi_1)) \, \left[S_1'(\eta) - S_2'(\eta)\right] = \big[p(S_3(\xi_1)) - p(S_2(\xi_1))  \big] \, \big[S_1'(\eta) - S_2'(\eta)\big].
\end{equation*}
Similarly, Assumption \ref{ass:nonlinearity}-\eqref{ass:mon}, see also Fig. \ref{plot:Fd}, implies that
$$
\alpha_{-} < S_1(\eta) < S_1(\xi_2) < S_2(\xi_2) < S_2(\eta) < S_3(\eta) < S_3(\xi_2) < \beta_{+}.
$$
Therefore, we find
$$
\sum_{j=1}^3 \chi_{S_1(\xi_2)}(S_j(\eta)) \, |S_j'(\eta)| =  \sum_{j=1}^3 \chi_{S_2(\xi_2)}(S_j(\eta)) \, |S_j'(\eta)| = S_1'(\eta),
$$
$$
\sum_{j=1}^3 \chi_{S_3(\xi_2)}(S_j(\eta)) \, |S_j'(\eta)| = S_1'(\eta) - S_2'(\eta) + S_3'(\eta).
$$
Hence, \eqref{easier_formula_for_S_and_R} and Remark \ref{rem:practical_connection_betkandl} imply
\begin{equation*}
\mathcal{R}(\eta,\xi_2) = q(\xi_2) \, S_1'(\eta) + p(S_3(\xi_2)) \, \left[S_3'(\eta) - S_2'(\eta)\right].
\end{equation*}
\end{proof}
\noindent Finally, we prove the following characterization result based on the nondegeneracy condition \eqref{ass:nondeg} in Assumption \ref{ass:nonlinearity}.

\begin{thm}[Strong convergence of $\veps$ for non-degenerate $F$] \label{thm:analytic_char_result}
Let $F$ be as in Assumption \ref{ass:nonlinearity} (in particular, it satisfies~\eqref{ass:nondeg} in this assumption). Let $p$ and $q$ be given by \eqref{eq:def_of_pq_weak_lim}.
Then, there exists $\alpha(t,x) \geq 0$ such that $q(t,x,\xi) = \chi_{\alpha(t,x)}(\xi)$, cf. equation \eqref{eq_general_kinetic_function}. Consequently $\alpha(t,x)=v(t,x)$ and $\veps(t,x)$ converges strongly in $L^2((0,T)\times\Omega)$.
\end{thm}
\begin{proof}
We know that $p(\xi)$ and $q(\xi)$ are bounded, nonnegative and compactly supported. Moreover, they vanish for $\xi < 0$. Consider the support of $q$ denoted with 
$\supp\, q$.
The proof is divided for three parts where we systematically increase possible support of $q$.\\

\noindent \underline{Case 1: $\supp\, {q}\subset [0,f_{-})$.} Consider $\eta \in (0,f_{-})$ and $\xi$ such that $\xi > \eta$. We want to use \eqref{thmeq:main_identity_loctx}. Notice that $S_1'(\eta) > 0$ and $S_2'(\eta) = S_3'(\eta) = 0$. Using \eqref{easier_formula_for_S_and_R} and Remark \ref{rem:practical_connection_betkandl} we write
$$
 \mathcal{S}(\eta) = p(S_1'(\eta)) \,S_1'(\eta)\, = q(\eta)\,S_1'(\eta).
$$
Moreover, $\chi_{\eta}(\xi) = 0$. When it comes to $\mathcal{R}(\eta,\xi)$ we observe that there is only one $\tau$ such that $F(\tau) = \eta$, namely $\tau = S_1(\eta)$. Moreover, as $\xi > \eta$, $S_1(\eta) < S_i(\xi)$ for $i = 1,2,3$. Therefore, \eqref{easier_formula_for_S_and_R} implies
$$
\mathcal{R}(\eta,\xi) = S_1'(\eta) \, \sum_{i=1}^3 (-1)^{i+1} \,
p(S_i(\xi)) \, \mathds{1}_{J_i}(\xi) = S_1'(\eta) \, q(\xi).
$$
Hence, \eqref{thmeq:main_identity_loctx} simplifies to
$$
q(\xi)\,q(\eta)\,S_1'(\eta) = q(\xi) \, S_1'(\eta) +q(\xi) - q(\xi)\, q(\eta)
$$
which can be rearranged to
$$
\big(q(\xi)\,q(\eta) - q(\xi)\big) \, \big(S_1'(\eta)+1 \big) = 0.
$$
As $S_1'(\eta) > 0$, this implies that for $\eta \in (0,f_{-})$ and $\xi > \eta$ we have
$$
q(\xi) \, (q(\eta) - 1) = 0.
$$
Since $q\leq 1$ is non-increasing, it follows that $q(\xi) =1$ with at most one jump from 1 to 0.
If there is a jump, the result is proved and thus we now continue with the case ${q} = 1$ on $(0,f_-)$.
\\

\noindent \underline{Case 2: $[0,f_{-}) \subset \supp\, {q} \subset [0,f_{+})$.} We consider three points $\xi_1$, $\eta$ and $\xi_2$ such that $f_{-} < \xi_1 < \eta <  \xi_2 < f_{+}$. The proof in this case will be concluded if we demonstrate
\begin{equation}\label{eq:mainclaimforstep2}
 q(\xi_2) \neq 0 \implies q(\xi_1) = 1.
\end{equation} 
Using \eqref{thmeq:main_identity_loctx} with $(\xi_1,\eta)$ and $(\xi_2,\eta)$ we obtain two equations
\begin{equation}\label{prf_thm_eq_case1_lxi12}
\left[q(\xi_1) - 1 \right] \, \mathcal{S}(\eta) =   \mathcal{R}(\eta,\xi_1) + q(\eta) - q(\eta)\,q(\xi_1) = \mathcal{R}(\eta,\xi_1) + q(\eta)\,(1 - q(\xi_1)),
\end{equation}
\begin{equation}\label{prf_thm_eq_case1_lxi122}
 q(\xi_2) \, \mathcal{S}(\eta) =   \mathcal{R}(\eta,\xi_2) + q(\xi_2) - q(\eta)\,q(\xi_2).
\end{equation}
We multiply \eqref{prf_thm_eq_case1_lxi12} with $q(\xi_2)$ and combine it with \eqref{prf_thm_eq_case1_lxi122} to deduce
$$
\big(q(\xi_1) - 1\big)\, \big(\mathcal{R}(\eta,\xi_2) +q(\xi_2) \big) = \mathcal{R}(\eta,\xi_1)\,q(\xi_2).
$$
Now, we use Lemma \ref{R_unstable_region}. Namely, we plug \eqref{eq:formulaforRetaxi1} and \eqref{eq:formulaforRetaxi2} above to discover identity
\begin{multline*}
\big(q(\xi_1) - 1\big)   \big[ q(\xi_2) S_1'(\eta) + p(S_3(\xi_2)) \, \big(S_3'(\eta) - S_2'(\eta)\big)+ q(\xi_2)\big] = \\ = q(\xi_2) \,\big(p(S_3(\xi_1)) - p(S_2(\xi_1))  \big) \, \big(S_1'(\eta) - S_2'(\eta)\big).
\end{multline*}
It can be rewritten as
\begin{multline*}
\big(q(\xi_1) - 1\big) \, \big[q(\xi_2) \, \big(S_1'(\eta) + 1\big) + p(S_3(\xi_2)) \, \big(S_3'(\eta) + 1 - S_2'(\eta) - 1\big)\big] = \\ = q(\xi_2) \,\big(p(S_3(\xi_1)) - p(S_2(\xi_1))  \big) \, \big(S_1'(\eta) + 1 - S_2'(\eta) - 1\big).
\end{multline*}
This can be seen as a linear equation for functions $S_1'(\eta)+1$, $S_2'(\eta)+1$ and $S_3'(\eta)+1$ satisfied for $\eta \in (\xi_1,\xi_2)$. Using \eqref{ass:nondeg} in Assumption~\ref{ass:nonlinearity} we obtain that sum of the coefficients standing next to these functions vanish. Hence,
$$
\big(q(\xi_1) - 1\big) \, q(\xi_2) = 0 \qquad (\xi_1 < \xi_2)
$$
which proves \eqref{eq:mainclaimforstep2}.
\iffalse
{\color{red} Let us write explicitly coefficients standing next to these functions:
\begin{align*}
S_1'(\eta)+1: \qquad & a_1 := \big(q(\xi_1) - 1\big) \, q(\xi_2) - q(\xi_2) \,\big(p(S_3(\xi_1)) - p(S_2(\xi_1))  \big), \\
S_2'(\eta)+1: \qquad & a_2 := -\big(q(\xi_1) - 1\big) \,p(S_3(\xi_2)) + q(\xi_2) \,\big(p(S_3(\xi_1)) - p(S_2(\xi_1))  \big), \\
S_3'(\eta)+1: \qquad & a_3 := \big(q(\xi_1) - 1\big) \,p(S_3(\xi_2)).
\end{align*}
These coefficients sum up to $\big(q(\xi_1) - 1\big) \, q(\xi_2)$. Therefore, with $R$ and $S$ from condition~\eqref{ass:nondeg} in Assumption~\ref{ass:nonlinearity}, we obtain
\begin{equation}\label{eq:main_ident_replace_nonlin_ind}
0 = a_1 + R\, a_2 + S\, a_3 = \big(q(\xi_1) - 1\big) \, q(\xi_2) + (R-1)\, a_2 + (S-1)\, a_3. 
\end{equation}
We can write 
$$
(R-1)\, a_2 + (S-1)\, a_3 = (S-R)\,\big(q(\xi_1) - 1\big) \,p(S_3(\xi_2)) + (R-1)\,q(\xi_2) \,\big(p(S_3(\xi_1)) - p(S_2(\xi_1))  \big).
$$
Both terms above are nonnegative because 
$$
S \geq R \geq 1, \qquad  0 \leq q(\xi_1), q(\xi_2) \leq 1, \qquad p(S_3(\xi_1)) \leq p(S_2(\xi_1))
$$
due to the monotonicity of $p$, cf. \eqref{X3} in Lemma~\ref{lem:prop_k_and_l}. Hence, from \eqref{eq:main_ident_replace_nonlin_ind}, we deduce
$$
\big(q(\xi_1) - 1\big) \, q(\xi_2) \geq 0, \qquad \qquad (f_{-} < \xi_1  <  \xi_2 < f_{+}).
$$
As the opposite inequality is clear, we have equality above and this proves \eqref{eq:mainclaimforstep2}.}
\fi\\

\noindent  \underline{Case 3: $[0,f_{+}) \subset \supp\,q$.} This is very similar to the first case. We consider $\eta \in (f_{+}, \infty)$ and arbitrary $\xi \in (0, \eta)$. Note that  $\chi_{\eta}(\xi) = 1$. Moreover, $S_1'(\eta) = S_2'(\eta) = 0$ and so, $$
 \mathcal{S}(\eta) = p(S_3'(\eta)) \,S_3'(\eta)\, = q(\eta)\,S_3'(\eta).
$$
Finally, $ S_i(\xi) < S_3(\eta)$ for $i =1,2,3$. Therefore $\mathcal{R}(\eta,\xi) = 0$ and so, \eqref{thmeq:main_identity_loctx} simplifies to 
$$
q(\xi) \, q(\eta) \, S_3'(\eta) = S_3'(\eta) \, q(\eta) + q(\eta) - q(\xi) \,q(\eta) .
$$
Since $S_3'(\eta) > 0$, we deduce that for all $\eta \in (f_+, \infty)$ and $\xi < \eta$ we have
$$
q(\eta)\left(q(\xi) - 1 \right) = 0.
$$
We conclude as in Case 1 and find that $q(\xi)= \chi_{\alpha}(\xi)$.
\\

\noindent Once we know that, it is easy to derive strong convergence. Since 
$$
v(t,x) \weaksinv \veps(t,x)= \int \chi_{\veps(t,x)}(\xi) \diff \xi  \weaks \int_{\R} q(t,x,\xi) \diff \xi =\int_{R} \chi_{\alpha(t,x)}(\xi) \diff \xi = \alpha(t,x),
$$
we find that $\alpha(t,x)= v(t,x)$. Then, we have 
$$
(\veps(t,x))^2= 2\int \xi \, \chi_{\veps(t,x)}(\xi) \diff \xi  \weaks  2 \int  \xi \, \chi_{v(t,x)}(\xi) \diff \xi = (v(t,x))^2
$$
which implies strong convergence.
\end{proof}

\noindent Now, we may conclude the proof of Theorem \ref{thm_main_id_YM}.
\begin{proof}[Proof of Theorem \ref{thm_main_id_YM}]
Using \eqref{eq:kandlequation} in Lemma~\ref{lem:prop_k_and_l}, we may write
$$
\chi_{v(t,x)}(\xi)= \int p(t,x,\tau)\, \frac{\partial}{\partial \tau} \chi_{F(\tau)}(\xi) \diff \tau
=- \int \frac{\partial p}{\partial \tau}(t,x,\tau) \, \chi_{F(\tau)}(\xi) \diff \tau.
$$
Differentiating with respect to $\xi$ we obtain
$$
\delta_{v(t,x)}(\xi) = - \int \frac{\partial p}{\partial \tau}(t,x,\tau) \,  \delta_{\xi=F(\tau)}(\xi) \diff \tau.
$$
Recalling that $\frac{\partial p}{\partial \tau}$ is nonpositive, for $F(\tau) \neq v(t,x)$ we conclude that $\frac{\partial p}{\partial \tau}(t,x,\tau)=0$. In other words, $p(t,x,\tau)$ can only have non-increasing jumps at the three roots of $F(\tau)=v(t,x)$, i.e., $S_1(v(t,x))$, $S_2(v(t,x))$ and $S_3(v(t,x))$, see Fig.\ref{plot:Fd}. Finally, because $p(\tau)$ decreases from 1 to $0$, the three weights $\{\lambda_i(t,x)\}_{i=1,2,3}$ have to sum-up to $1$ and the representation formula for  $u$ in Theorem \ref{thm_main_id_YM} is proved.
\end{proof}

%As always, let $p$ and $q$ be the weak$^*$ limits of kinetic functions $p^{\varepsilon}(t,x,\xi) = %\chi_{\ueps(t,x)}(\xi)$ and $q^{\varepsilon}(t,x,\xi) = \chi_{\veps(t,x)}(\xi)$. We know that
%$$q(t,x,\xi) = \chi_{\alpha(t,x)}(\xi).$$
%It follows that
%\begin{equation}\label{eq:derivativeoflYM}
%{\partial_\xi}\, q(t,x,\xi) = \delta_0 - \delta_{\alpha(t,x)}.
%\end{equation}
%Recall that we write $\{\mu_{t,x}\}_{t,x}$ and $\{\nu_{t,x}\}_{t,x}$ for Young measures generated by sequences $\{\ueps\}_{\varepsilon \in (0,1)}$ and $\{\veps\}_{\varepsilon \in (0,1)}$ respectively. In  the language of Young measures, thanks to Lemma~\ref{lem:YM_con_with_kf} %\eqref{eq:derivativeoflYM} can be rewritten as
%we may write
%$$\nu_{t,x} = F^{\#}\mu_{t,x} = \delta_{v(t,x)}.$$
%This implies that $\alpha(t,x) = v(t,x)$ where $v$ is a weak$^*$ limit of $\veps$. Moreover, as %Young measure generated by $\{\veps\}_{\varepsilon \in (0,1)}$ collapses to a point mass, $\veps %\to v$ strongly in $L^2((0,T)\times\Omega)$.\\
%
%\noindent Now, we know that $F^{\#}\mu_{t,x} = \delta_{v(t,x)}$. 
%Therefore $\mu_{t,x}$ can be concentrated only on the set $F^{-1}(v(t,x)) = \{S_1(v(t,x)), S_2(v(t,x)),S_3(v(t,x))\}$, see  \ref{plot:Fd}. It follows that there are $\lambda_1(t,x)$, $\lambda_2(t,x)$ and $\lambda_3(t,x)$ such that
%\begin{equation}\label{proof_repr_of_gamma}
%\gamma_{t,x} = \lambda_1(t,x)\,\delta_{S_1(v(t,x))} + \lambda_2(t,x)\,\delta_{S_2(v(t,x))} +
% \lambda_3(t,x)\,\delta_{S_3(v(t,x))} \end{equation}
%where $\lambda_1(t,x) + \lambda_2(t,x) + \lambda_3(t,x) = 1$.

%----------------------------
\section{Equation satisfied by weights $\lambda_1(t,x)$, $\lambda_2(t,x)$ and $\lambda_3(t,x)$}\label{sect:diff_ineq_weights}
%----------------------------
\noindent In order to prove Theorem \ref{thm:intro_eqns_weights}, we first connect the Young measure representation \eqref{eq:repr_main_result} from Theorem~\ref{thm_main_id_YM} with the kinetic function $p$. Due to Lemma \ref{lem:YM_con_with_kf}, for fixed $(t,x)$, function $p(t,x,\xi)$ has four jumps (see Fig.~\ref{plot:K}):
\begin{itemize}
    \item from 0 to 1 at $\xi = 0$,
    \item from 1 to $\kappa_1(t,x)$ at $\xi = S_1(v(t,x))$ for some $\kappa_1(t,x) \in (0,1)$,
    \item from $\kappa_1(t,x)$ to $\kappa_2(t,x)$ at $\xi = S_2(v(t,x))$ for some $\kappa_2(t,x) \in (0,\kappa_1(t,x))$,
    \item from $\kappa_2(t,x)$ to 0 at $\xi = S_3(v(t,x))$.
\end{itemize}
%%%%%%%%%%%%%%%%%%%%%%%%%%%%%%%%%%%
%%%%%%%%%%%%%%%%%%%%%%%%%%%%%%%%%%%
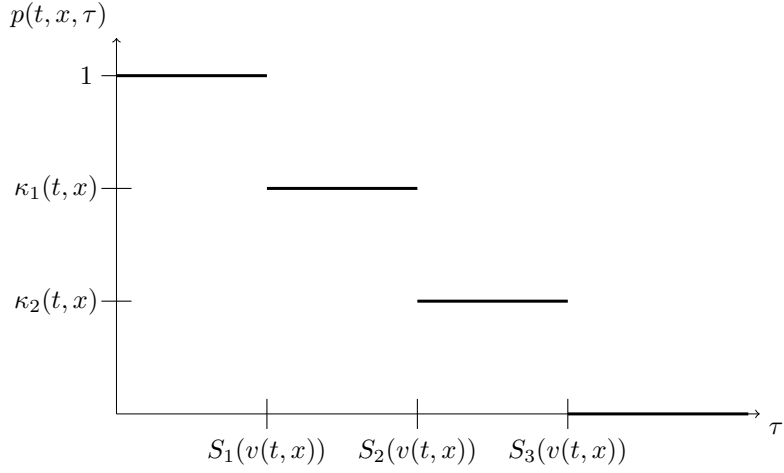
\begin{figure}
\begin{tikzpicture}

%coordinate system and dashed lines
\draw[line width=0.1mm,->] (0,0) -- (8.55,0) node[anchor=north west] {$\tau$};
\draw[line width=0.1mm,->] (0,0) -- (0,5) node[anchor=south east] {$p(t,x,\tau)$};

%dashed lines with f- and f+
\draw [line width=0.4mm](0,4.5) -- (2,4.5);
\draw [line width=0.4mm](2,3) -- (4,3);
\draw [line width=0.4mm](4,1.5) -- (6,1.5);
\draw [line width=0.4mm](6,0) -- (8.4,0);
\node at (-0.4, 4.5) {$1$};
\node at (-0.8, 3) {$\kappa_1(t,x)$};
\node at (-0.8, 1.5) {$\kappa_2(t,x)$};
\draw (-0.2,4.5) -- (0.2,4.5);
\draw (-0.2,3) -- (0.2,3);
\draw (-0.2,1.5) -- (0.2,1.5);

%nodes on x axis
\node at (2,-0.5) {$S_1(v(t,x))$};
\node at (4,-0.5) {$S_2(v(t,x))$};
\node at (6,-0.5) {$S_3(v(t,x))$};
\draw (2,-0.2) -- (2,0.2);
\draw (4,-0.2) -- (4,0.2);
\draw (6,-0.2) -- (6,0.2);

\end{tikzpicture}
\vspace{-6mm}
\caption{Plot of the function $\tau \mapsto p(t,x,\tau)$ for fixed $(t,x)$.}
\label{plot:K}
\end{figure}
\noindent Once again from Lemma \ref{lem:YM_con_with_kf} we deduce that
$$
\lambda_1(t,x) = 1-\kappa_1(t,x), \qquad \lambda_2(t,x) = \kappa_1(t,x) - \kappa_2(t,x), \qquad \lambda_3(t,x) = \kappa_2(t,x).
$$
Hence, to understand dynamics of weights $\lambda_1(t,x)$, $\lambda_2(t,x)$ and $\lambda_3(t,x)$, it is sufficient to study coefficients $\kappa_1(t,x)$ and $\kappa_2(t,x)$. Moreover, we have representation
\begin{equation}\label{eq:formula_for_p}
p(t,x,\tau) = \mathds{1}_{[0,S_1(v(t,x))]}(\tau) + \kappa_1(t,x) \, \mathds{1}_{[S_1(v(t,x)),S_2(v(t,x))]}(\tau) + \kappa_2(t,x) \, \mathds{1}_{[S_2(v(t,x)),S_3(v(t,x))]}(\tau).
\end{equation} 
\begin{proof}[Proof of \eqref{thm:part1_eq_weights} in Theorem \ref{thm:intro_eqns_weights}]
Due to assumptions, $\partial_t v(t,x)$ exists in the Sobolev sence. Therefore, we can write \begin{equation}\label{eq:formula_for_dpdt}
\begin{split}
\partial_t p(t,x,\tau) =&- \delta_{\tau = S_1(v(t,x))}(\tau) \, S_1'(v(t,x)) \, \partial_t v(t,x) \phantom{\Big(\Big)}\\
&-\Big(\delta_{\tau = S_2(v(t,x))}(\tau) \, S_2'(v(t,x)) - \delta_{\tau = S_1(v(t,x))}(\tau) \, S_1'(v(t,x)) \Big)\,\kappa_1(t,x)\, \partial_t v(t,x)\\
&-\Big(\delta_{\tau = S_3(v(t,x))}(\tau) \, S_3'(v(t,x)) - \delta_{\tau = S_2(v(t,x))}(\tau) \, S_2'(v(t,x)) \Big)\,\kappa_2(t,x)\, \partial_t v(t,x) \\
&+\partial_t\kappa_1(t,x) \, \mathds{1}_{[S_1(v(t,x)),S_2(v(t,x))]}(\tau) + \partial_t\kappa_2(t,x) \, \mathds{1}_{[S_2(v(t,x)),S_3(v(t,x))]}(\tau) \phantom{\Big(\Big)}
\end{split}
\end{equation}
in the sense of distributions. Moreover, functions $p(t,x,\xi)$ and $q(t,x,\xi)$ satisfy PDE \eqref{eq:PDE2_sat_by_kin_LIMIT}. We note that for all $i = 1,2,3$ it holds
$$
\int_{\R} \delta_{\tau = S_1(v(t,x))}(\tau) \, \delta_{\xi = F(\tau)}(\xi) \diff \tau = \delta_{\xi = v(t,x)}(\xi)
$$
because $F(S_i(v(t,x))) = v(t,x)$. Therefore, plugging \eqref{eq:formula_for_dpdt} into \eqref{eq:PDE2_sat_by_kin_LIMIT}, we deduce
\begin{equation}\label{eq:ODE_combined_noloc}
\begin{split}
\left[\int_{S_1(v(t,x))}^{S_2(v(t,x))} \delta_{\xi=F(\tau)}(\xi) \diff \tau\right]& \, \partial_t \kappa_1(t,x) + \left[\int_{S_2(v(t,x))}^{S_3(v(t,x))} \delta_{\xi=F(\tau)}(\xi) \diff \tau\right] \, \partial_t \kappa_2(t,x)
=\partial_{\xi} n(t,x,\xi)\\& + (\Delta - \partial_t)\, q(t,x,\xi)  + \delta_{\xi = v(t,x)}(\xi) \, \partial_t v(t,x)\, S_1'(v(t,x)) \phantom{\int_{S_2(v(t,x))}^{S_3(v(t,x))}}\\ & + \delta_{\xi = v(t,x)}(\xi)\,\sum_{i=1}^2  (S_{i+1}'(v(t,x))-S_i'(v(t,x))) \, \kappa_i(t,x). \phantom{\int_{S_2(v(t,x))}^{S_3(v(t,x))}}
\end{split}
\end{equation}
Since sequences $\{\partial_t \veps\}_{\varepsilon \in (0,1)}$ and $\{\Delta \veps\}_{\varepsilon \in (0,1)}$ are uniformly bounded in $L^2((0,T)\times\Omega)$, we deduce from equation \eqref{system_2} that \begin{equation}\label{eq:convergence_v_to_F_faster}
\| \veps - F(\ueps)\|_{L^2((0,T)\times\Omega)} \leq C\, \varepsilon
\end{equation}
for some constant $C$. Therefore, term $n_2^{\varepsilon}(t,x,\xi)$ defined with \eqref{eq:term_n2_measure} converges to 0 as $\varepsilon \to 0$ and measure $n(t,x,\xi)$ equals
\begin{equation}\label{eq:form_of_n_special_case}
n(t,x,\xi) = \delta_{v(t,x)}(\xi)\, |\nabla v(t,x)|^2.
\end{equation}
We claim now that \eqref{eq:formula_for_dpdt} implies 
\begin{equation}\label{eq:weights_smooth_claim_kappa}
\partial_t \kappa_1(t,x) =0, \qquad \qquad \partial_t \kappa_2(t,x) = 0
\end{equation}
which is equivalent to the assertion because $\lambda_1 + \lambda_2 + \lambda_3 = 1$. To see \eqref{eq:weights_smooth_claim_kappa}, fix $(t_0,x_0)$ such that $v(t_0,x_0) \in (f_{-}, f_{+})$ and open neighbourhood $\mathcal{O}_{t_0,x_0}$ such that $v(t,x) < \xi_0 < f_{+}$ in $\mathcal{O}_{t_0,x_0}$ for some $\xi_0$. Then, \eqref{eq:ODE_combined_noloc} for $(t,x,\xi) \in \mathcal{O}_{t_0,x_0} \times (\xi_0, f_{+})$ boils down to
$$
\left[\int_{S_1(v(t,x))}^{S_2(v(t,x))} \delta_{\xi=F(\tau)}(\xi) \diff \tau\right] \, \partial_t \kappa_1(t,x) + \left[\int_{S_2(v(t,x))}^{S_3(v(t,x))} \delta_{\xi=F(\tau)}(\xi) \diff \tau\right] \, \partial_t \kappa_2(t,x) = 0
$$
because $q(t,x,\xi) = \chi_{v(t,x)}(\xi)$. However, as $v(t,x)<\xi$, the second term vanishes because 
$$
S_2(\xi) < S_2(v(t,x)) < S_3(v(t,x)) < S_3(\xi),
$$
cf. Fig.~\ref{plot:Fd}. Hence, we deduce $\partial_t \kappa_1(t,x) = 0$ in $\mathcal{O}_{t_0,x_0}$. Equality $\partial_t \kappa_3(t,x) = 0$ follows from the similar reasoning - this time we need to localize equation so that $f_{-}<\xi < v(t,x)$.
\end{proof}
\begin{rem}
One can prove \eqref{thm:part1_eq_weights} in Theorem \ref{thm:intro_eqns_weights} under weaker assumption on $\{\Delta \veps\}_{\varepsilon \in (0,1)}$, namely that the sequence $\{\varepsilon^{1/2 - \delta}\,\Delta \veps\}_{\varepsilon \in (0,1)}$ is bounded in $L^2((0,T)\times\Omega)$ for some $\delta > 0$. Indeed, in this case we obtain
$$
\| \veps - F(\ueps)\|_{L^2((0,T)\times\Omega)} \leq C\, \varepsilon^{1/2+\delta}
$$
instead of \eqref{eq:convergence_v_to_F_faster} and the same conclusion as in \eqref{eq:form_of_n_special_case} follows concerning form of the measure $n(t,x,\xi)$. The assumption on $\{\partial_t \veps \}_{\varepsilon \in (0,1)}$ is still necessary to guarantee existence of $\partial_t v$.
\end{rem}

\noindent We remark that identity \eqref{eq:ODE_combined_noloc} was obtained in \cite[Appendix A]{MR2103092}. We also note that derivation of~\eqref{eq:ODE_combined_noloc} requires some smoothness of $v(t,x)$ and so, our PDE for kinetic functions~\eqref{eq:PDE2_sat_by_kin_LIMIT} may be seen as a weak formulation of~\eqref{eq:ODE_combined_noloc}.

\begin{proof}[Proof of \eqref{thm:part2_eq_weights} in Theorem \ref{thm:intro_eqns_weights}.]
This time we need to be more careful as $\partial_t v$ can be understood only in the sense of distributions and computation \eqref{eq:formula_for_dpdt} is no longer valid. Still we can write
\begin{equation}\label{eq:formula_for_dpdt_special_case}
\begin{split}
\partial_t p(t,x,\tau) = \,& \partial_t \mathds{1}_{[0,S_1(v(t,x))]}(\tau)+ \partial_t\big(\kappa_1(t,x) \, \mathds{1}_{[S_1(v(t,x)),S_2(v(t,x))]}(\tau)\big) \\ &+\partial_t\big( \kappa_2(t,x) \, \mathds{1}_{[S_2(v(t,x)),S_3(v(t,x))]}(\tau)\big).
\end{split}
\end{equation}
Therefore, plugging \eqref{eq:formula_for_dpdt_special_case} into \eqref{eq:PDE2_sat_by_kin_LIMIT}, we deduce
\begin{equation}\label{eq:PDE_with_p_plugged_gen_case}
\begin{split}
&\partial_t \int_{0}^{S_1(v(t,x))}  \delta_{\xi = F(\lambda)}(\xi) \diff \lambda + \partial_t \int_{S_1(v(t,x))}^{S_2(v(t,x))} \kappa_1(t,x) \, \delta_{\xi = F(\lambda)}(\xi) \diff \lambda \,+ \\ & \qquad  + \partial_t \int_{S_2(v(t,x))}^{S_3(v(t,x))} \kappa_2(t,x) \, \delta_{\xi = F(\lambda)}(\xi) \diff \lambda + \partial_t q(t,x,\xi) - \Delta_x q(t,x,\xi) =  \partial_{\xi}\, {n}(t,x,\xi).
\end{split}
\end{equation}
As in the proof of \eqref{thm:part1_eq_weights} above, we localize around 
$(t,x,\xi)$ such that $v(t,x)<\xi$. In this case,
$$
S_1(v(t,x)) < S_1(\xi) < S_2(\xi) < S_2(v(t,x)) < S_3(v(t,x)) < S_3(\xi)
$$
so that $\int_{0}^{S_1(v(t,x))}  \delta_{\xi = F(\lambda)}(\xi) \diff \lambda = 0$ and $\int_{S_2(v(t,x))}^{S_3(v(t,x))} \kappa_2(t,x) \, \delta_{\xi = F(\lambda)}(\xi) \diff \lambda = 0$. Therefore,
$$
\partial_t \int_{S_1(v(t,x))}^{S_2(v(t,x))} \kappa_1(t,x) \, \delta_{\xi = F(\lambda)}(\xi) \diff \lambda  =  \partial_{\xi}\, {n}(t,x,\xi)
$$
holds in any open set where $f_{-} < v(t,x) < \xi_0 < \xi < f_{+}$ for some $\xi_0$. This proves \eqref{eq:intro_thm_PDE_weigts_1}. In a similar way we localize around $\xi < v(t,x)$ and deduce \eqref{eq:intro_thm_PDE_weigts_2}.  
\end{proof}
\noindent Understanding dynamics of weights and deducing more information from~\eqref{eq:intro_thm_PDE_weigts_1} and \eqref{eq:intro_thm_PDE_weigts_2} is one of the three open problems discussed in Section~\ref{sect:open_problems}.
%---------------------------
\section{Concluding remarks}
%----------------------------

\subsection{Connection with Plotnikov's approach}
\label{sect:connection_with_Plotnikov}

\noindent Our work uses ideas from a seminal paper by Plotnikov \cite{MR1299852} who studied another regularization 
\begin{equation}\label{eq:sect7_reg_for_back}
\partial_t \weps = \Delta A(\weps) + \varepsilon \, \Delta (\partial_t \weps)
\end{equation}
of the forward-backward problems $\partial_t w = \Delta A(w)$ where $A$ has a similar monotonicity profile as  presented  in Fig.~\ref{plot:Fd} for function $F$. Below we summarize his argument adapted to our system \eqref{system_1}--\eqref{system_2} and emphasize the differences between his and our approach. We remark that Plotnikov worked with Young measures and obtained identities for measures of arbitrary sets rather than functional identities as \eqref{thmeq:main_identity}. Nevertheless, we believe that the equivalent approach of kinetic formulation allows to simplify the reasoning and bring new information.\\

\noindent Here, following \cite{MR1299852}, we assume additionally that $F'(u) > -1$ and we define functions
\begin{equation}\label{eq:def_of_I_and_A}
I(u):= u + F(u), \qquad A(w) = F(I^{-1}(w)).
\end{equation}
Note that $I$ is bijective so function $A$ has the same monotonicity profile as function $F$. If we let $\weps = \ueps + \veps$, we deduce from \eqref{system_1}--\eqref{system_2} that
\begin{equation}\label{eq:sum_of_u_and_v}
\partial_t \weps = \Delta \veps
\end{equation}
and the connection between \eqref{eq:sect7_reg_for_back} and \eqref{eq:sum_of_u_and_v} comes from an observation that $\veps - A(\weps) \to 0$ in $L^2((0,T)\times\Omega)$. Indeed, it is sufficient to write
\begin{equation}\label{eq:strong_conv_v_AW}
\veps - A(\weps) = \veps - A(\ueps + F(\ueps)) + A(\ueps + F(\ueps)) - A(\weps)
\end{equation}
and use the strong convergence of $\veps - F(\ueps) \to 0$ from Corollary \ref{cor:converg_A_v}.\\

\noindent Plotnikov works in variables $(\weps, A(\weps))$ rather than with $(\ueps, \veps)$ as in this paper. Let
$$
k^{\varepsilon}(t,x,\xi) = \chi_{\weps(t,x)}(\xi), \qquad \qquad l^{\varepsilon}(t,x,\xi) = \chi_{A(w^{\varepsilon})(t,x)}(\xi).
$$
To obtain a PDE satisfied by the weak$^*$ limits of the kinetic functions $k^{\varepsilon}$ and $l^{\varepsilon}$ as in Theorem \ref{thm:2ndPDE_kinfunct}, Plotnikov introduces functions
\begin{equation}\label{cc_notation_compact_functions}
G(\lambda) = \int_{0}^{\lambda} g(A(\tau)) \diff \tau, \qquad \qquad  H(\lambda) = \int_0^{A(\lambda)} h(\tau) \diff \tau.
\end{equation}
where $g$ and $h$ are smooth test functions. Using chain rule and \eqref{eq:sum_of_u_and_v}, we obtain
\begin{align*}
\partial_t  G(\weps) &= g(A(\weps)) \, \Delta \veps = g(\veps)\, \Delta \veps + \big(g(A(\weps)) - g(\veps)\big) \, \Delta \veps \\ &= 
\Delta_x \widetilde{G}(\veps) -\,g'(\veps) \, |\nabla \veps|^2  + \big(g(A(\weps)) - g(\veps)\big) \, \Delta \veps.
\end{align*}
where $\widetilde{G}$ is a primitive function of $g$. This leads to PDE
\begin{equation}\label{eq:PDE_kin_func_previous}
{\partial_t}\,  \int_{\R}  k(t,x,\tau) \, \delta_{\xi = A(\tau)}(\xi) \diff \tau - \Delta_x\, l(t,x,\xi) = {\partial_\xi}\, m(t,x,\xi).
\end{equation}
where $k$ and $l$ are weak$^*$ limits of functions $k^{\varepsilon}$ and $l^{\varepsilon}$ while $m$ is a weak$^*$ limit of the sequence
$$
m^{\varepsilon}(t,x,\xi) = \delta_{\xi = \veps(t,x)}(\xi)\, |\nabla \veps(t,x)|^2 +  \int_0^1 \delta_{\xi = s\,A(\weps) + (1-s)\, \veps}(\xi) \diff s \, \left(\veps - A(\weps) \right) \, \Delta \veps.
$$
It is a little bit surprising that it is not clear what is the sign of measure $m$ while the measure $n$ from PDE \eqref{eq:PDE2_sat_by_kin_LIMIT} is nonnegative. It is even more mysterious if we realize that left-hand sides of limiting equations \eqref{eq:PDE2_sat_by_kin_LIMIT} and \eqref{eq:PDE_kin_func_previous} are exactly the same. This is the content of the following lemma.
\begin{lem}\label{lem:two_PDEs_coincide}
Let $k$, $l$, $p$ and $q$ be the weak$^*$ limits of kinetic functions $k^{\varepsilon}(t,x,\xi) = \chi_{\weps(t,x)}(\xi)$, $l^{\varepsilon}(t,x,\xi) = \chi_{A(\weps)(t,x)}(\xi)$, $p^{\varepsilon}(t,x,\xi) = \chi_{\ueps(t,x)}(\xi)$ and $q^{\varepsilon}(t,x,\xi) = \chi_{\veps(t,x)}(\xi)$. Then,
\begin{enumerate}
\item\label{C1} $l(t,x,\xi) = q(t,x,\xi)$,
\item\label{C112} $k(t,x,\xi) = p(t,x,I^{-1}(\xi))$,
\item\label{C2} $\int_{\R}  k(t,x,\tau) \, \delta_{\xi = A(\tau)}(\xi) \diff \tau = \int_{\R} p(t,x,\lambda) \, \delta_{\xi = F(\lambda)}(\xi) \diff \lambda +  q(t,x,\xi)$.
\end{enumerate}
In particular, \eqref{C1} and \eqref{C2} imply that left-hand sides of PDEs \eqref{eq:PDE2_sat_by_kin_LIMIT} and \eqref{eq:PDE_kin_func_previous} coincide.
\end{lem}
\begin{proof}
Property \eqref{C1} follows from strong convergence of $\veps - A(\weps) \to 0$ cf. \eqref{eq:strong_conv_v_AW}. To see \eqref{C112}, we first observe that $I(\ueps) - \weps \to 0$ strongly. Hence,
$$
k(t,x,\xi) = \int_{\R} q(t,x,\lambda) \, I'(\lambda) \, \delta_{\xi = I(\lambda)}(\xi) \diff \lambda
$$
in the sense of distributions. Note that $I$ is assumed to be invertible so we may transform this distributional identity into the pointwise one. For any test function $\psi(\xi)$ we have
\begin{multline*}
\int_{\R} \psi(\xi) \, k(t,x,\xi) \diff \xi = \int_{\R} q(t,x,\lambda) \, I'(\lambda) \, \psi(I(\lambda)) \diff \lambda = \\ = \int_{\R} q(t,x,I^{-1}(\xi)) \, I'(I^{-1}(\xi)) \, \psi(\xi) \, (I^{-1})'(\xi) \diff \xi = 
 \int_{\R} \psi(\xi) \, q(t,x,I^{-1}(\xi)) \diff \xi.
\end{multline*}
To prove \eqref{C2}, we note that function $A$ has three inverses
$
R_i(\xi):= I(S_i(\xi))$ where $1 \leq i \leq 3
$. By chain rule,
$$
R_i'(\xi) = I'(S_i(\xi)) \, S_i'(\xi) = (1 + F'(S_i(\xi)) ) \,S_i'(\xi) = S_i'(\xi) + \mathds{1}_{J_i}(\xi)
$$
because $S_i$ are inverses of $F$ on $J_i$ and so, $F'(S_i(\xi)) \,S_i'(\xi) = \mathds{1}_{J_i}(\xi)$. Moreover, by virtue of Lemma \ref{lem:int_by_subst_part1} and Corollary \ref{rem:practical_connection_betkandl}, we have  
$$
l(t,x,\xi) = \sum_{i=1}^3 (-1)^{i+1}\,k(t,x,R_i(\xi)) \,\mathds{1}_{J_i}(\xi), \quad \int_{\R} k(t,x,\tau) \, \delta_{\xi=A(\tau)}(\xi) \diff \tau = \sum_{i=1}^3 k(t,x,R_i(\xi)) \,|R_i'(\xi)|.
$$
Using these facts, we write
\begin{align*}
\int_{\R} k(t,x,\tau) \, \delta_{\xi=A(\tau)}(\xi) \diff \tau &= \sum_{i=1}^{3} (-1)^{i+1}\, k(t,x,R_i(\xi)) \,  R_i'(\xi) =\\
& = \sum_{i=1}^{3} (-1)^{i+1}\, k(t,x,R_i(\xi)) \,  S_i'(\xi) + \sum_{i=1}^{3} (-1)^{i+1}\, k(t,x,R_i(\xi)) \mathds{1}_{J_i}(\xi) =\\
& = \sum_{i=1}^3 (-1)^{i+1}\, p(t,x,I^{-1}(R_i(\xi))) \, S_i'(\xi) + l(t,x,\xi) \\
& = \sum_{i=1}^3 (-1)^{i+1}\, p(t,x,S_i(\xi))) \, S_i'(\xi) + q(t,x,\xi) \\
& = \int_{\R} p(t,x,\tau) \, \delta_{\xi = F(\tau)}(\xi) \diff \tau + q(t,x,\xi).
\end{align*}
\end{proof}
\noindent Lemma \ref{lem:two_PDEs_coincide} implies that the next steps in identifiction of the limit in our paper and in the work of Plotnikov are equivalent. However, working directly with $(\ueps, \veps)$ allows to formulate equation for kinetic function with nonnegative measure. This is useful to gain more information on weights $\lambda_1(t,x)$, $\lambda_2(t,x)$ and $\lambda_3(t,x)$ from equations \eqref{eq:intro_thm_PDE_weigts_1} and \eqref{eq:intro_thm_PDE_weigts_2}. Another advantage of our approach is that it does not require assumption $F'(u) > -1$.

\subsection{Open problems and further perspectives}\label{sect:open_problems}
\noindent We list here three problems connected to our work. For now, their treatment seems to be unavailable for us.\\

\noindent \underline{Problem 1: fast reaction limit for general reaction-diffusion system.} System \eqref{system_1}--\eqref{system_2} studied in this paper is a special case of
\begin{align}
\partial_t \ueps &= d_1\, \Delta \ueps + \frac{\veps - F(\ueps)}{\varepsilon}, \label{system_1_gen}\\
\partial_t \veps &= d_2\, \Delta \veps + \frac{F(\ueps) - \veps}{\varepsilon} \label{system_2_gen}
\end{align}
for some $d_1, d_2 \geq 0$. Using refined energy estimates from \cite{morita2016reaction}, fast reaction limit was established in \cite[Theorem 2.9]{MR3908864} for two special cases
\begin{equation}\label{eq:cond_d1_d2}
d_2 \geq d_1,\, F'(u) + \frac{d_1}{d_2} > 0, \qquad \mbox{ or } \qquad d_1 > d_2,\, F'(u) + \frac{d_2}{d_1} > 0.
\end{equation}
More precisely, it was proved that $\weps := \ueps + \veps$ converges strongly to the solution of 
\begin{equation}\label{eq:limit_problem_some_special_case}
\partial_t w - \Delta A(w) = 0, \qquad \frac{\partial}{\partial \textbf{n}} w = 0
\end{equation}
where
$$
A(w) = d_1 u + d_2 F(u)\quad  \mbox{ with }\quad  w = u + F(u).
$$
Function $A$ is well-defined because conditions \eqref{eq:cond_d1_d2} imply that $F'(u)>-1$. Limiting equation \eqref{eq:limit_problem_some_special_case} is a consequence of summing up \eqref{system_1_gen}--\eqref{system_2_gen} together with a priori estimates that gives strong convergence of $\ueps + \veps \to u+v$ and $\veps - F(\ueps) \to 0$. However, if one only assumes $F'(u) > -1$ without \eqref{eq:cond_d1_d2}, the only available energy estimate is 
\begin{equation*}
\begin{split}
&\frac{\diff}{\diff t} \int_{\Omega} \left[ \widetilde{F}(\ueps) + \frac{1}{2}\left(\veps\right)^2 + \varepsilon \, d_1 |\nabla \ueps|^2 + \frac{d_1^2 + d_1\,d_2}{2\,(d_2^2 - d_1^2)} \left(w^{\varepsilon}\right)^2 \right]\diff x = \\
& \qquad \qquad = - \varepsilon \int_{\Omega} \left(d_1 \, \Delta \ueps + \frac{\veps - F(\ueps)}{\varepsilon} \right)^2 \diff x - \frac{1}{d_2 - d_1} \int_{\Omega} \left|d_1 \nabla \ueps + d_2 \nabla \veps \right|^2 \diff x
\end{split}
\end{equation*}
where $\widetilde{F}$ is a primitive function of $F$. This equality is too weak to deduce any strong convergence. The only result we can prove in that case is that, seting $\weps = \ueps + \veps$ and $\zeps = d_1\ueps + d_2\veps$, we have
$$
\weps \weaks w := u+v, \qquad \zeps \weaks z := d_1u + d_2v, \qquad \veps - F(\ueps) \weaks 0, \qquad w_t = \Delta z
$$
but it is not clear at all what is the coupling between functions $w$ and $z$.\\

\noindent \underline{Problem 2: nondegeneracy condition.} Strong convergence $\veps \to v$ in our work is rather unavailable to be obtained from a priori estimates. It is a consequence of careful analysis of kinetic function (or Young measure) in Theorem \ref{thm:analytic_char_result} and nondegeneracy condition \eqref{ass:nondeg} in Assumption \ref{ass:nonlinearity}. This technical assumption excludes piecewise affine functions $F$ and is hard to verify for particular examples as one needs to know inverses of $F$ explicitly. On the other hand, this type of condition is a common assumption in papers concerning mostly regularization of forward-backward parabolic problems \cite{MR1015926, MR1299852} but also some hyperbolic equations with nonmonotone model functions \cite{MR657784}. We would like to know whether nondegeneracy assumption can be waived and if not, what happens with solutions to \eqref{system_1}--\eqref{system_2} in the case of piecewise affine function $F$.\\

\noindent \underline{Problem 3: understanding equation on weights $\lambda_1$, $\lambda_2$ and $\lambda_3$.} In Section \ref{sect:diff_ineq_weights} we proved equations \eqref{eq:intro_thm_PDE_weigts_1} and \eqref{eq:intro_thm_PDE_weigts_2} that carry some information on the weights in the decomposition \eqref{eq:repr_main_result}. It is not clear what is the information hidden in this equality. For instance, is it possible to determine asymptotic values of $\{\lambda_i(t,x)\}_{i=1,2,3}$? Some information can be gained from the sign of $n$. For example, for $\psi(\xi) \geq 0$, we can test \eqref{eq:intro_thm_PDE_weigts_1} with $\Psi(\xi) = \int_0^{\xi} \psi(\eta) \diff \eta$ to deduce
$$
\partial_t \left[ \kappa_1(t,x) \, \int_{S_1(v(t,x))}^{S_2(v(t,x))} \Psi(F(\lambda)) \diff \lambda \right] = - \int_{\R} \psi(\xi) \diff m(t,x,\xi) \leq 0
$$
so that the function $t \mapsto \kappa_1(t,x) \, \int_{S_1(v(t,x))}^{S_2(v(t,x))} \Psi(F(\lambda)) \diff \lambda$ is nonincreasing. However, it is not clear how $\kappa_1(t,x)$ interacts with $\int_{S_1(v(t,x))}^{S_2(v(t,x))} \Psi(F(\lambda)) \diff \lambda$ to gain more information from that.
\bibliographystyle{abbrv}
\bibliography{fastlimit}

\begin{thebibliography}{10}

\bibitem{MR657784}
G.~Andrews and J.~M. Ball.
\newblock Asymptotic behaviour and changes of phase in one-dimensional
  nonlinear viscoelasticity.
\newblock {\em J. Differential Equations}, 44(2):306--341, 1982.
\newblock Special issue dedicated to J. P. LaSalle.

\bibitem{MR1721985}
C.~M. Bender and S.~A. Orszag.
\newblock {\em Advanced mathematical methods for scientists and engineers.
  {I}}.
\newblock Springer-Verlag, New York, 1999.
\newblock Asymptotic methods and perturbation theory, Reprint of the 1978
  original.

\bibitem{MR3466547}
M.~Bertsch, F.~Smarrazzo, and A.~Tesei.
\newblock Pseudo-parabolic regularization of forward-backward parabolic
  equations: power-type nonlinearities.
\newblock {\em J. Reine Angew. Math.}, 712:51--80, 2016.

\bibitem{MR2009623}
D.~Bothe and D.~Hilhorst.
\newblock A reaction-diffusion system with fast reversible reaction.
\newblock {\em J. Math. Anal. Appl.}, 286(1):125--135, 2003.

\bibitem{MR3005532}
D.~Bothe, M.~Pierre, and G.~Rolland.
\newblock Cross-diffusion limit for a reaction-diffusion system with fast
  reversible reaction.
\newblock {\em Comm. Partial Differential Equations}, 37(11):1940--1966, 2012.

\bibitem{MR4051984}
M.~Burger, P.~Friele, and J.-F. Pietschmann.
\newblock On a reaction-cross-diffusion system modeling the growth of
  glioblastoma.
\newblock {\em SIAM J. Appl. Math.}, 80(1):160--182, 2020.

\bibitem{MR3783102}
J.~A. Carrillo, Y.~Huang, and M.~Schmidtchen.
\newblock Zoology of a nonlocal cross-diffusion model for two species.
\newblock {\em SIAM J. Appl. Math.}, 78(2):1078--1104, 2018.

\bibitem{MR1981403}
G.-Q. Chen and B.~Perthame.
\newblock Well-posedness for non-isotropic degenerate parabolic-hyperbolic
  equations.
\newblock {\em Ann. Inst. H. Poincar\'{e} Anal. Non Lin\'{e}aire},
  20(4):645--668, 2003.

\bibitem{MR3501846}
E.~C.~M. Crooks and D.~Hilhorst.
\newblock Self-similar fast-reaction limits for reaction-diffusion systems on
  unbounded domains.
\newblock {\em J. Differential Equations}, 261(3):2210--2250, 2016.

\bibitem{MR4040718}
E.~S. Daus, L.~Desvillettes, and A.~J\"{u}ngel.
\newblock Cross-diffusion systems and fast-reaction limits.
\newblock {\em Bull. Sci. Math.}, 159:102824, 29, 2020.

\bibitem{MR589954}
L.~C. Evans.
\newblock A convergence theorem for a chemical diffusion-reaction system.
\newblock {\em Houston J. Math.}, 6(2):259--267, 1980.

\bibitem{MR1034481}
L.~C. Evans.
\newblock {\em Weak convergence methods for nonlinear partial differential
  equations}, volume~74 of {\em CBMS Regional Conference Series in
  Mathematics}.
\newblock Published for the Conference Board of the Mathematical Sciences,
  Washington, DC; by the American Mathematical Society, Providence, RI, 1990.

\bibitem{MR2103092}
L.~C. Evans and M.~Portilheiro.
\newblock Irreversibility and hysteresis for a forward-backward diffusion
  equation.
\newblock {\em Math. Models Methods Appl. Sci.}, 14(11):1599--1620, 2004.

\bibitem{GST2019}
B.~Gess, J.~Sauer, and E.~Tadmor.
\newblock Optimal regularity in time and space for the porous medium equation.
\newblock {\em arXiv preprint arXiv:1902.08632}, 2019.

\bibitem{MR1814364}
D.~Gilbarg and N.~S. Trudinger.
\newblock {\em Elliptic partial differential equations of second order}.
\newblock Classics in Mathematics. Springer-Verlag, Berlin, 2001.
\newblock Reprint of the 1998 edition.

\bibitem{MR2251792}
M.~Iida, M.~Mimura, and H.~Ninomiya.
\newblock Diffusion, cross-diffusion and competitive interaction.
\newblock {\em J. Math. Biol.}, 53(4):617--641, 2006.

\bibitem{MR3655798}
M.~Iida, H.~Monobe, H.~Murakawa, and H.~Ninomiya.
\newblock Vanishing, moving and immovable interfaces in fast reaction limits.
\newblock {\em J. Differential Equations}, 263(5):2715--2735, 2017.

\bibitem{MR3905643}
M.~Iida, H.~Ninomiya, and H.~Yamamoto.
\newblock A review on reaction-diffusion approximation.
\newblock {\em J. Elliptic Parabol. Equ.}, 4(2):565--600, 2018.

\bibitem{MR1201239}
P.-L. Lions, B.~Perthame, and E.~Tadmor.
\newblock A kinetic formulation of multidimensional scalar conservation laws
  and related equations.
\newblock {\em J. Amer. Math. Soc.}, 7(1):169--191, 1994.

\bibitem{MR3329327}
A.~Marciniak-Czochra.
\newblock Reaction-diffusion-{ODE} models of pattern formation.
\newblock In {\em Evolutionary equations with applications in natural
  sciences}, volume 2126 of {\em Lecture Notes in Math.}, pages 387--438.
  Springer, Cham, 2015.

\bibitem{MR3039206}
A.~Marciniak-Czochra, G.~Karch, and K.~Suzuki.
\newblock Unstable patterns in reaction-diffusion model of early
  carcinogenesis.
\newblock {\em J. Math. Pures Appl. (9)}, 99(5):509--543, 2013.

\bibitem{MR3600397}
A.~Marciniak-Czochra, G.~Karch, and K.~Suzuki.
\newblock Instability of {T}uring patterns in reaction-diffusion-{ODE} systems.
\newblock {\em J. Math. Biol.}, 74(3):583--618, 2017.

\bibitem{MR2563628}
C.~Mascia, A.~Terracina, and A.~Tesei.
\newblock Two-phase entropy solutions of a forward-backward parabolic equation.
\newblock {\em Arch. Ration. Mech. Anal.}, 194(3):887--925, 2009.

\bibitem{MR2646071}
Y.~Morita and T.~Ogawa.
\newblock Stability and bifurcation of nonconstant solutions to a
  reaction-diffusion system with conservation of mass.
\newblock {\em Nonlinearity}, 23(6):1387--1411, 2010.

\bibitem{morita2016reaction}
Y.~Morita and N.~Shinjo.
\newblock Reaction-diffusion models with a conservation law and pattern
  formations.
\newblock {\em Josai Mathematical Monographs}, 9:177--190, 2016.

\bibitem{MR3908864}
A.~Moussa, B.~Perthame, and D.~Salort.
\newblock Backward parabolicity, cross-diffusion and {T}uring instability.
\newblock {\em J. Nonlinear Sci.}, 29(1):139--162, 2019.

\bibitem{murakawa2019fast}
H.~Murakawa.
\newblock Fast reaction limit of reaction-diffusion systems.
\newblock {\em arXiv preprint arXiv:1901.04278}, 2019.

\bibitem{MR2776460}
H.~Murakawa and H.~Ninomiya.
\newblock Fast reaction limit of a three-component reaction-diffusion system.
\newblock {\em J. Math. Anal. Appl.}, 379(1):150--170, 2011.

\bibitem{MR894077}
F.~Murat.
\newblock A survey on compensated compactness.
\newblock In {\em Contributions to modern calculus of variations ({B}ologna,
  1985)}, volume 148 of {\em Pitman Res. Notes Math. Ser.}, pages 145--183.
  Longman Sci. Tech., Harlow, 1987.

\bibitem{MR1015926}
A.~Novick-Cohen and R.~L. Pego.
\newblock Stable patterns in a viscous diffusion equation.
\newblock {\em Trans. Amer. Math. Soc.}, 324(1):331--351, 1991.

\bibitem{MR2367252}
M.~Otsuji, S.~Ishihara, C.~Co, K.~Kaibuchi, A.~Mochizuki, and S.~Kuroda.
\newblock A mass conserved reaction-diffusion system captures properties of
  cell polarity.
\newblock {\em PLoS Comput. Biol.}, 3(6):1040--1054, 2007.

\bibitem{MR1620640}
V.~Padr\'{o}n.
\newblock Sobolev regularization of a nonlinear ill-posed parabolic problem as
  a model for aggregating populations.
\newblock {\em Comm. Partial Differential Equations}, 23(3-4):457--486, 1998.

\bibitem{MR1452107}
P.~Pedregal.
\newblock {\em Parametrized measures and variational principles}, volume~30 of
  {\em Progress in Nonlinear Differential Equations and their Applications}.
\newblock Birkh\"{a}user Verlag, Basel, 1997.

\bibitem{MR2064166}
B.~Perthame.
\newblock {\em Kinetic formulation of conservation laws}, volume~21 of {\em
  Oxford Lecture Series in Mathematics and its Applications}.
\newblock Oxford University Press, Oxford, 2002.

\bibitem{MR1099693}
B.~Perthame and E.~Tadmor.
\newblock A kinetic equation with kinetic entropy functions for scalar
  conservation laws.
\newblock {\em Comm. Math. Phys.}, 136(3):501--517, 1991.

\bibitem{MR1299852}
P.~I. Plotnikov.
\newblock Passage to the limit with respect to viscosity in an equation with a
  variable direction of parabolicity.
\newblock {\em Differentsial' nye Uravneniya}, 30(4):665--674, 734, 1994.

\bibitem{MR3014456}
T.~Roub\'{\i}\v{c}ek.
\newblock {\em Nonlinear partial differential equations with applications},
  volume 153 of {\em International Series of Numerical Mathematics}.
\newblock Birkh\"{a}user/Springer Basel AG, Basel, second edition, 2013.

\bibitem{MR3071451}
T.~O. Sakamoto.
\newblock Hopf bifurcation in a reaction-diffusion system with conservation of
  mass.
\newblock {\em Nonlinearity}, 26(7):2027--2049, 2013.

\bibitem{MR2429862}
F.~Smarrazzo.
\newblock On a class of equations with variable parabolicity direction.
\newblock {\em Discrete Contin. Dyn. Syst.}, 22(3):729--758, 2008.

\bibitem{MR584398}
L.~Tartar.
\newblock Compensated compactness and applications to partial differential
  equations.
\newblock In {\em Nonlinear analysis and mechanics: {H}eriot-{W}att
  {S}ymposium, {V}ol. {IV}}, volume~39 of {\em Res. Notes in Math.}, pages
  136--212. Pitman, Boston, Mass.-London, 1979.

\bibitem{MR2765690}
A.~Terracina.
\newblock Qualitative behavior of the two-phase entropy solution of a
  forward-backward parabolic problem.
\newblock {\em SIAM J. Math. Anal.}, 43(1):228--252, 2011.

\bibitem{vanag2009cross}
V.~K. Vanag and I.~R. Epstein.
\newblock Cross-diffusion and pattern formation in reaction--diffusion systems.
\newblock {\em Physical Chemistry Chemical Physics}, 11(6):897--912, 2009.

\end{thebibliography}
\end{document}